\documentclass[11pt]{article}%
\usepackage{amsxtra,amscd}
\usepackage{graphicx}
\usepackage{amsmath}
\usepackage{amsfonts}
\usepackage{amssymb}
\usepackage[all]{xy}%
\usepackage{url}
\usepackage[textwidth=5.5in,textheight=9.0in,centering]{geometry}
\usepackage[amsmath,hyperref,thmmarks]{ntheorem}
\usepackage{natbib}
\usepackage[scaled=0.92]{helvet}	
\usepackage{verbatim}
\usepackage[all]{xy}

\renewcommand{\tilde}{\widetilde}
\def\1{{1\mkern-7mu1}}

\newcommand\bquote{\begin{quote}}
\newcommand\equote{\end{quote}}
\newcommand\bcomment{}
\newcommand\bsmall{\begin{small}}
\newcommand\esmall{\end{small}}
\newcommand\bfootnotesize{\begin{footnotesize}}
\newcommand\efootnotesize{\end{footnotesize}}

\newcommand{\textnf}{\textnormal}

\let\cite=\citealt

\newcommand{\eb}[1]{{\itshape\bfseries#1}{\index{#1}}}
\renewcommand{\emph}{\eb}

\renewcommand{\Gamma}{\varGamma}
\renewcommand{\Pi}{\varPi}
\renewcommand{\Sigma}{\varSigma}

\DeclareMathOperator{\Aut}{Aut}

\DeclareMathOperator{\End}{End}

\DeclareMathOperator{\Gal}{Gal}
\DeclareMathOperator{\GL}{GL}

\DeclareMathOperator{\Hom}{Hom}

\DeclareMathOperator{\Ker}{Ker}

\DeclareMathOperator{\MT}{MT}

\DeclareMathOperator{\ord}{ord}

\DeclareMathOperator{\Pic}{\mathrm{Pic}}

\DeclareMathOperator{\Tr}{Tr}

\newcommand{\Mot}{\mathsf{Mot}}

\theoremnumbering{arabic}
\theoremheaderfont{\scshape}
\RequirePackage{latexsym}
\newtheorem{X}{X}[section]
\theorembodyfont{\upshape}
\newtheorem{plain}[X]{}
\theorembodyfont{\slshape}

\theoremseparator{.}

\newtheorem*{conjecture}[X]{Conjecture}
\newtheorem*{conjectureT}[X]{Conjecture $T^r(X,\ell)$ (Tate)}
\newtheorem*{conjectureE}[X]{Conjecture $E^r(X,\ell)$ (Equality of equivalence relations)}
\newtheorem*{conjectureS}[X]{Conjecture $S^r(X,\ell)$ (Partial semisimplicity)}
\newtheorem{corollary}[X]{Corollary}
\newtheorem{lemma}[X]{Lemma}
\newtheorem{proposition}[X]{Proposition}
\newtheorem{theorem}[X]{Theorem}
\theorembodyfont{\upshape}
\newtheorem{definition}[X]{Definition}

\newtheorem{question}[X]{Question}
\newtheorem{remark}[X]{Remark}

\theorembodyfont{\small}

\newtheorem{aside}[X]{Aside}
\theoremstyle{nonumberplain}
\theorembodyfont{\normalsize}
\theoremsymbol{\ensuremath{_\Box}}
\RequirePackage{amssymb}
\newtheorem{proof}{Proof}
\newtheorem{pf}{Sketch of proof}
\qedsymbol{\ensuremath{_\Box}}
\theoremclass{LaTeX}
\begin{document}

\title{{\huge The Tate conjecture over finite fields (AIM talk)}}
\author{{\Large J.S. Milne}}
\date{}
\maketitle

\begin{abstract}
These are my notes for a talk at the The Tate Conjecture workshop at the
American Institute of Mathematics in Palo Alto, CA, July 23--July 27, 2007,
somewhat revised and expanded. The intent of the talk was to review what is
known and to suggest directions for research.\newline v1 (September 19, 2007):
first version on the web. \newline v2 (\today): revised and expanded.

\end{abstract}

\tableofcontents

\subsubsection{Conventions}

\renewcommand{\thefootnote}{\fnsymbol{footnote}} \footnotetext{MSC2000: 14C25,
11G25} \renewcommand{\thefootnote}{\arabic{footnote}}

All varieties are smooth and projective. Complex conjugation on $\mathbb{C}{}$
is denoted by $\iota$. The symbol $\mathbb{F}{}$ denotes an algebraic closure
of $\mathbb{F}{}_{p}$, and $\ell$ always denotes a prime $\neq p$. On the
other hand, $l$ is allowed to equal $p$. For a variety $X$, $H^{\ast
}(X)=\bigoplus\nolimits_{i}H^{i}(X)$ and $H^{2\ast}(X)(\ast)=\bigoplus
\nolimits_{i}H^{2i}(X)(i)$; both are graded algebras. I denote a canonical (or
a specifically given) isomorphism by $\simeq$. I assume that the reader is
familiar with the basic theory of abelian varieties as, for example, in
\cite{milne1986ab}.

\section{The conjecture, and some folklore}

Let $X$ be a variety over $\mathbb{F}{}$. A model $X_{1}$ of $X$ over a finite
subfield $k_{1}$ of $\mathbb{F}{}$ gives rise to a commutative diagram:%
\[
\begin{CD}
Z^r(X) @>{c^{r}}>> H^{2r}(X,\mathbb{Q}_{\ell}(r))\\
@AAA@AAA\\
Z^r(X_{1}) @>{c^{r}}>>H^{2r}(X_{1},\mathbb{Q}_{\ell}(r)).
\end{CD}
\]
Here $Z^{r}(\ast)$ denotes the group of algebraic cycles of codimension $r$ on
a variety (free $\mathbb{Z}{}$-module generated by the closed irreducible
subvarieties of codimension $r$) and $c^{r}$ is the cycle map. The image of
the vertical arrow at right is contained in $H^{2r}(X,\mathbb{Q}{}_{\ell
}(r))^{\Gal(\mathbb{F}{}/k_{1})}$ and $Z^{r}(X)=\varinjlim_{X_{1}/k_{1}}%
Z^{r}(X_{1})$, and so the image of the top cycle map is contained in
\[
H^{2r}(X,\mathbb{Q}{}_{\ell}(r))^{\prime}\,\,\overset{\text{{\tiny def}}}%
{=}\,\,\bigcup\nolimits_{X_{1}/k_{1}}H^{2r}(X,\mathbb{Q}{}_{\ell
}(r))^{\Gal(\mathbb{F}{}/k_{1})}.
\]
In his talk at the AMS Summer Institute at Woods Hole in July, 1964, Tate
conjectured the following:\footnote{Tate's talk is included in the
mimeographed proceedings of the conference, which were distributed by the AMS
to only a select few. Despite their great historical importance --- for
example, they contain the only written account by Artin and Verdier of their
duality theorem, and the only written account by Serre and Tate of their
lifting theorem --- the AMS has ignored requests to make the proceedings more
widely available. Fortunately, Tate's talk was reprinted in the proceedings of
an earlier conference (Arithmetical algebraic geometry. Proceedings of a
Conference held at Purdue University, December 5--7, 1963. Edited by O. F. G.
Schilling, Harper \& Row, Publishers, New York 1965).}

\begin{conjectureT}
The $\mathbb{Q}{}_{\ell}$-vector space $H^{2r}(X,\mathbb{Q}{}_{\ell
}(r))^{\prime}$ is spanned by algebraic classes.
\end{conjectureT}

\noindent The conjecture implies that, for any model $X_{1}/k_{1}$, the
$\mathbb{Q}{}_{\ell}$-subspace $H^{2r}(X,\mathbb{Q}{}_{\ell}%
(r))^{\Gal(\mathbb{F}{}/k_{1})}$ is spanned by the classes of algebraic cycles
on $X_{1}$; conversely, if this is true for all models $X_{1}/k_{1}$ over
(sufficiently large) finite fields $k_{1}$, then $T^{r}(X,\ell)$ is true.

I write $T^{r}(X)$ (resp. $T(X,\ell)$, resp. $T(X)$) for the conjecture that
$T^{r}(X,\ell)$ is true for all $\ell$ (resp. all $r$, resp. all $r$ and
$\ell$).

In the same talk, Tate mentioned the following \textquotedblleft conjectural
statement\textquotedblright:

\begin{conjectureE}
The kernel of the cycle class map $c^{r}\colon Z^{r}(X)\rightarrow
H^{2r}(X_{\mathbb{F}{}},\mathbb{Q}{}_{\ell}(r))$ consists exactly of the
cycles numerically equivalent to zero.
\end{conjectureE}

\noindent Both conjectures are existence statements for algebraic classes. It
is well known that Conjecture $E^{1}(X)$ holds for all $X$ (see
\cite{tate1994}, \S 5).

Let $X$ be a variety over $\mathbb{F}{}$. The choice of a model of $X$ over a
finite subfield of $\mathbb{F}{}$ defines a Frobenius map $\pi\colon
X\rightarrow X$. For example, for a model $X_{1}\subset\mathbb{P}{}^{n}$ over
$\mathbb{F}{}_{q}$, $\pi$ acts as%
\[
(a_{1}\colon a_{2}\colon\ldots)\mapsto(a_{1}^{q}\colon a_{2}^{q}\colon
\ldots)\colon X_{1}(\mathbb{F}{})\rightarrow X_{1}(\mathbb{F}{}),\quad
X_{1}(\mathbb{F}{})\simeq X(\mathbb{F}{}).
\]
Any such map will be called a \emph{Frobenius map} of $X$ (or a $q$%
-\emph{Frobenius map} if it is defined by a model over $\mathbb{F}{}_{q}$). If
$\pi_{1}$ and $\pi_{2}$ are $p^{n_{1}}$- and $p^{n_{2}}$-Frobenius maps of
$X$, then $\pi_{1}^{n_{2}N}=\pi_{2}^{n_{1}N}$ for some $N>1$.\footnote{Because
any two models of $X$ become isomorphic over a finite subfield of
$\mathbb{F}{}$; when the model $X_{1}/k_{1}$ is replaced by $X_{1K}/K$, then
its Frobenius map $\pi$ is replaced by $\pi^{\lbrack K\colon k_{1}]}$.} For a
Frobenius map $\pi$ of $X$, we use a subscript $a$ to denote the generalized
eigenspace with eigenvalue $a$, i.e., $\bigcup\nolimits_{N}\Ker\left(
(\pi-a)^{N}\right)  .$

\begin{conjectureS}
Every Frobenius map $\pi$ of $X$ acts semisimply on $H^{2r}(X,\mathbb{Q}%
{}_{\ell}(r))_{1}$ (i.e., it acts as $1$).
\end{conjectureS}

Weil proved that, for an abelian variety $A$ over $\mathbb{F}{}$, the
Frobenius maps act semisimply on $H^{1}(A,\mathbb{Q}{}_{\ell})$. It follows
that they act semisimply on all the cohomology groups $H^{i}(A,\mathbb{Q}%
{}_{\ell})\simeq\bigwedge\nolimits^{i}H^{1}(A,\mathbb{Q}{}_{\ell})$. In
particular, Conjecture $S(X)$ holds when $X$ is an abelian variety over
$\mathbb{F}{}$.

From now on, I'll write $\mathcal{T}{}_{\ell}^{r}(X)$ for $H^{2r}%
(X,\mathbb{Q}{}{}_{\ell}(r))^{\prime}$ and call its elements the \emph{Tate
classes} of degree $r$ on $X$. Note that $\mathcal{T}{}_{\ell}^{\ast
}(X)\overset{\text{{\tiny def}}}{=}\bigoplus\nolimits_{r}\mathcal{T}{}_{\ell
}^{r}(X)$ is a graded $\mathbb{Q}{}_{\ell}$-subalgebra of $H^{2\ast
}(X,\mathbb{Q}{}_{\ell}(\ast))$.

\begin{aside}
\label{t1}One can ask whether the Tate conjecture holds integrally, i.e.,
whether the map%
\begin{equation}
c^{r}\colon Z^{r}(X)\otimes\mathbb{Z}{}_{\ell}\rightarrow H^{2r}%
(X_{\mathbb{F}{}},\mathbb{Z}{}_{\ell}(r))^{\prime} \label{e1}%
\end{equation}
is surjective for all varieties $X$ over $\mathbb{F}{}$. Clearly
$H^{2r}(X_{\mathbb{F}{}},\mathbb{Z}{}_{\ell}(r))^{\prime}$ contains all
torsion classes, and essentially the same argument that shows that not all
torsion classes in Betti cohomology are algebraic, shows that not all torsion
classes in \'{e}tale cohomology are algebraic.\footnote{The proof shows that
the odd dimensional Steenrod operations are zero on the torsion algebraic
classes but not on all torsion cohomology classes.} However, I don't know of
any varieties over $\mathbb{F}{}$ for which the map (\ref{e1}) is not
surjective modulo torsion. It is known that if $T^{r}(X,\ell)$ and
$E^{r}(X,\ell)$ hold for a single $\ell$, then the map (\ref{e1}) is
surjective for all but possibly finitely many $\ell$. See \cite{milneR2004}, \S 3.
\end{aside}

\subsection{Folklore}

\noindent The next three theorems are folklore.

\begin{theorem}
\label{t2}Let $X$ be a variety over $\mathbb{F}{}$ of dimension $d$, and let
$r\in\mathbb{N}{}$. The following statements are equivalent:

\begin{enumerate}
\item $T^{r}(X,\ell)$ and $E^{r}(X,\ell)$ are true for a single $\ell$.

\item $T^{r}(X,\ell)$, $S^{r}(X,\ell)$, and $T^{d-r}(X,\ell)$ are true for a
single $\ell$.

\item $T^{r}(X,\ell)$, $E^{r}(X,\ell)$, $S^{r}(X,\ell)$, $T^{d-r}(X,\ell)$,
and $E^{d-r}(X,\ell)$ are true for all $\ell$, and the $\mathbb{Q}{}$-subspace
$\mathcal{A}_{\ell}^{r}(X)$ of $\mathcal{T}{}_{\ell}^{r}(X)$ generated by the
algebraic classes is a $\mathbb{Q}{}$-structure on $\mathcal{T}_{\ell}^{r}%
(X)$, i.e., $\mathcal{A}{}_{\ell}^{r}(X)\otimes_{\mathbb{Q}{}}\mathbb{Q}%
{}_{\ell}\simeq\mathcal{T}{}_{\ell}^{r}(X)$.

\item the order of the pole of the zeta function $Z(X,t)$ at $t=q^{-r}$ is
equal to the rank of the group of numerical equivalence classes of algebraic
cycles of codimension $r$.
\end{enumerate}
\end{theorem}

\noindent The proof is explained in \cite{tate1994}, \S 2.

\begin{theorem}
\label{t3}Let $X$ be a variety over $\mathbb{F}{}$ of dimension $d$. If
$S^{2d}(X\times X,\ell)$ is true, then every Frobenius map $\pi$ acts
semisimply on $H^{\ast}(X,\mathbb{Q}_{\ell})$.
\end{theorem}

\begin{proof}
If $a$ occurs as an eigenvalue of $\pi$ on $H^{r}(X,\mathbb{Q}{}_{\ell})$,
then $1/a$ occurs as an eigenvalue of $\pi$ on $H^{2d-r}(X,\mathbb{Q}{}_{\ell
}(d))$ (by Poincar\'{e} duality), and%
\[
H^{r}(X,\mathbb{Q}{}_{\ell})_{a}\otimes H^{2d-r}(X,\mathbb{Q}{}_{\ell
}(d))_{1/a}\subset H^{2d}(X\times X,\mathbb{Q}{}_{\ell}(d))_{1}%
\]
(K\"{u}nneth formula), from which the claim follows.
\end{proof}

An $\ell$\emph{-adic} \emph{Tate }$q$-\emph{structure} is a finite dimensional
$\mathbb{Q}{}_{\ell}$-vector space together with a linear (Frobenius) map
$\pi$ whose characteristic polynomial has rational coefficients and whose
eigenvalues are \emph{Weil }$q$\emph{-numbers}, i.e., algebraic numbers
$\alpha$ such that, for some integer $m$ called the \emph{weight} of $\alpha$,
$|\rho(\alpha)|=q^{m/2}$ for every homomorphism $\rho\colon\mathbb{Q}{}%
[\alpha]\rightarrow\mathbb{C}{}$ and, for some integer $n$, $q^{n}\alpha$ is
an algebraic integer. When the eigenvalues of $\pi$ are all algebraic
integers, the Tate structure is said to be \emph{effective}. For example, for
a variety $X$ over $\mathbb{F}{}_{q}$, $H^{r}(X_{\mathbb{F}{}},\mathbb{Q}%
{}_{\ell}(s))$ is a Tate $q$-structure of weight $r-2s$, which is effective if
$s=0$.

A Tate $p^{n_{1}}$-structure $\pi_{1}$ and a Tate $p^{n_{2}}$-structure
$\pi_{2}$ on a $\mathbb{Q}{}_{\ell}$-vector space $V$ are said to be
equivalent if $\pi_{1}^{n_{1}N}=\pi_{2}^{n_{2}N}$ for some $N$. This is an
equivalence relation, and an $\ell$\emph{-adic Tate structure} is a finite
dimensional $\mathbb{Q}{}_{\ell}$-vector space together with an equivalence
class of Tate $q$-structures. For example, for a variety $X$ over
$\mathbb{F}{}$, $H^{r}(X,\mathbb{Q}{}_{\ell}(s))$ is a Tate structure of
weight $r-2s$, which is effective if $s=0$.

Let $X$ be a smooth projective variety over $\mathbb{F}{}$. For each $r$, let
$F_{a}^{r}H^{i}(X,\mathbb{Q}{}_{\ell})$ denote the subspace of $H^{i}%
(X,\mathbb{Q}_{\ell})$ of classes with support in codimension $r$, i.e.,%
\[
F_{a}^{r}H^{i}(X,\mathbb{Q}{}_{\ell})=\bigcup\nolimits_{U}\Ker(H^{i}%
(X,\mathbb{Q}_{\ell})\rightarrow H^{i}(U,\mathbb{Q}{}_{\ell}))
\]
where $U$ runs over the open subvarieties of $X$ such that $X\smallsetminus U$
has codimension $\geq r$. If $Z=X\smallsetminus U$ has codimension $r$ and
$\tilde{Z}\rightarrow Z$ is a desingularization of $Z$, then%
\[
H^{i-2r}(\tilde{Z},\mathbb{Q}{}_{\ell})(-r)\rightarrow H^{i}(X,\mathbb{Q}%
{}_{\ell})\rightarrow H^{i}(U,\mathbb{Q}{}_{\ell})
\]
is exact (see \cite{deligne1974t}, 8.2.8; a similar proof applies to \'{e}tale
cohomology). This shows that $F_{a}^{r}H^{i}(X,\mathbb{Q}{}_{\ell})$ is an
effective Tate substructure of $H^{i}(X,\mathbb{Q}{}_{\ell})$ such that
$F_{a}^{r}H^{i}(X,\mathbb{Q}{}_{\ell})(r)$ is still effective.

\begin{conjecture}
[Generalized Tate conjecture]\label{t4}For a smooth projective variety $X$
over $\mathbb{F}{}$, every Tate substructure $V\subset H^{i}(X,\mathbb{Q}%
{}_{\ell})$ such that $V(r)$ is effective is contained in $F_{a}^{r}%
H^{i}(X,\mathbb{Q}{}_{\ell})$ (cf. \cite[10.3]{grothendieck1968}%
).\nocite{grothendieck1968}
\end{conjecture}

\begin{theorem}
\label{t5}\noindent Let $X$ be a variety over $\mathbb{F}{}$. If the Tate
conjecture holds for all varieties of the form $A\times X$ with $A$ an abelian
variety (and some $\ell$), then the generalized Tate conjecture holds for $X$
(and the same $\ell$).
\end{theorem}

\noindent The proof is explained in \cite{milneR2006}, 1.10.

\begin{remark}
\label{t6}There are $p$-analogues of all of the above conjectures and
statements. Let $W(k{})$ be the ring of Witt vectors with coefficients in a
perfect field $k$, and let $B(k)$ be its field of fractions. Let $\sigma$ be
the automorphism of $W(k)$ (or $B(k)$) that acts as $x\mapsto x^{p}$ on $k$.
For a variety $X$ over $k$, let $H_{p}^{r}(X)$ denote the crystalline
cohomology group with coefficients in $B(k)$. It is a finite dimensional
$B(k)$-vector space with a $\sigma$-linear Frobenius map $F$. Define
\[
\mathcal{T}{}_{p}^{r}(X)=\bigcup\nolimits_{X_{1}/k_{1}}H_{p}^{2r}%
(X_{1})(r)^{F^{n_{1}}=1}\quad\text{(}p^{n_{1}}=|k_{1}|\text{).}%
\]
This is a finite dimensional $\mathbb{Q}{}_{p}$-vector space, which the Tate
conjecture $T{}^{r}(X,p)$ says is spanned by algebraic classes.
\end{remark}

\subsection{\noindent Motivic interpretation}

Let $\Mot(\mathbb{F}{})$ be the category of motives over $\mathbb{F}{}$
defined using algebraic cycles modulo numerical equivalence. It is known that
$\Mot(\mathbb{F}{})$ is a semisimple Tannakian category (\cite{jannsen1992}).
Etale cohomology defines a functor $\omega_{\ell}$ on $\Mot(\mathbb{F}{})$ if
and only if Conjecture $E(X,\ell)$ holds for all varieties $X$. Assuming this,
conjecture $T(X,\ell)$ holds for all $X$ if and only if, for all $X$ and $Y$,
the image of the map%
\begin{equation}
\Hom(X,Y)\otimes\mathbb{Q}{}_{\ell}\rightarrow\Hom_{\mathbb{Q}{}_{\ell}%
}(\omega_{\ell}(X),\omega_{\ell}(Y)) \label{e2}%
\end{equation}
consists of the homomorphisms $\alpha\colon\omega_{\ell}(X)\rightarrow
\omega_{\ell}(Y)$ such that $\alpha\circ\pi_{X}=\pi_{Y}\circ\alpha$ for some
Frobenius maps $\pi_{X}$ and $\pi_{Y}$ of $X$ and $Y$ (necessarily
$q$-Frobenius maps for the same $q$). In other words, conjectures $E(X,\ell)$
and $T(X,\ell)$ hold for all $X$ if and only if $\ell$-adic \'{e}tale
cohomology defines an equivalence from $\Mot(\mathbb{F}{})\otimes
_{\mathbb{Q}{}}\mathbb{Q}{}_{\ell}$ to the category of $\ell$-adic Tate structures.

\section{Divisors on abelian varieties}

Tate (1966)\nocite{tate1966e} proved the Tate conjecture for divisors on
abelian varieties over $\mathbb{F}$,${}$ in the form:

\begin{theorem}
\label{t7}For all abelian varieties $A$ and $B$ over $\mathbb{F}{}_{q}$, the
map%
\begin{equation}
\Hom(A,B)\otimes\mathbb{Z}{}_{\ell}\rightarrow\Hom_{\mathbb{Q}{}_{\ell}%
}(T_{\ell}A,T_{\ell}B)^{\Gal(\mathbb{F}{}/\mathbb{F}{}_{q})} \label{e3}%
\end{equation}
is an isomorphism.
\end{theorem}

\begin{pf}
It suffices to prove this with $A=B$ (because $\Hom(A,B)$ is a direct summand
of $\End(A\times B)$). Choose a polarization on $A$, of degree $d^{2}$ say. It
defines a nondegenerate skew-symmetric form on $V_{\ell}A$, and a maximal
isotropic subspace $W$ of $V_{\ell}A$ stable under $\Gal(\mathbb{F}%
{}/\mathbb{F}{}_{q})$ will have dimension $g=\frac{1}{2}\dim V_{\ell}A$. For
$n\in\mathbb{N}{}$, let
\[
X(n)=T_{\ell}A\cap W+\ell^{n}T_{\ell}A\subset T_{\ell}A.
\]
For each $n$, there exists an abelian variety $A(n)$ and an isogeny
$A(n)\rightarrow A$ mapping $T_{\ell}A(n)$ isomorphically onto $X(n)$. There
are only finitely many isomorphism classes of abelian varieties in the set
$\{A(n)\}$ because each $A(n)$ has a polarization of degree $d^{2}$, and hence
can be realized as a closed subvariety of $\mathbb{P}{}^{3^{g}d-1}$ of degree
$3^{g}d(g!)$. Thus two of the $A(n)$'s are isomorphic, and we have constructed
a (nonobvious) isogeny. From this beginning, Tate was able to deduce the
theorem by exploiting the semisimplicity of the Frobenius map.
\end{pf}

\begin{corollary}
\label{t8}For varieties $X$ and $Y$ over $\mathbb{F}{}$,%
\[
T^{1}(X\times Y,\ell)\iff T^{1}(X,\ell)+T^{1}(Y,\ell).
\]

\end{corollary}

\begin{proof}
Compare the decomposition$\!$%
\[
\mathrm{NS}(X\times Y)\simeq\mathrm{NS}(X)\oplus\mathrm{NS}(Y)\oplus
\Hom(\mathrm{Alb}(X),\Pic^{0}(Y))
\]
with the similar decomposition of $H^{2}(X\times Y,\mathbb{Q}{}_{\ell}(1))$
given by the K\"{u}nneth formula.
\end{proof}

\begin{corollary}
\label{t9}The Tate conjecture $T^{1}(X)$ is true when $X$ is a product of
curves and abelian varieties over $\mathbb{F}{}$.
\end{corollary}

\begin{proof}
Let $A$ be an abelian variety over $\mathbb{F}{}$. Choose a polarization
$\lambda\colon A\rightarrow A^{\vee}$ of $A$, and let $^{\dagger}$ be the
Rosati involution on $\End^{0}(A)$ defined by $\lambda$. The map
$D\mapsto\lambda^{-1}\circ\lambda_{D}$ defines an isomorphism%
\[
NS(A)\otimes\mathbb{Q}{}\simeq\{\alpha\in\End^{0}(A)\mid\alpha^{\dagger
}=\alpha\}.
\]
Similarly,%
\[
\mathcal{T}{}_{\ell}^{1}(A)\simeq\{\alpha\in\End_{\mathbb{Q}{}_{\ell}}%
(V_{\ell}A)\mid\alpha^{\dagger}=\alpha\text{ and }\alpha\pi=\pi\alpha\text{
for some Frobenius map }\pi\},
\]
and so $T^{1}$ for abelian varieties follows from Theorem \ref{t7}. Since
$T^{1}$ is obvious for curves, the general statement follows from Corollary
\ref{t8}.
\end{proof}

As $E^{1}(X,\ell)$ is true for all varieties $X$, the equivalent statements in
Theorem \ref{t2} hold for products of curves and abelian varieties in the case
$r=1$. According to some more folklore (\cite{tate1994}, 5.2), $T^{1}(X)$ and
$E^{1}(X)$ hold for any variety $X$ for which there exists a dominant rational
map $Y\rightarrow X$ with $Y$ a product of curves and abelian varieties.

\begin{aside}
\label{t9a}The reader will have noted the similarity of (\ref{e2}) and
(\ref{e3}). Tate (1994) describes how he was led to his conjecture partly by
his belief that (\ref{e3}) was true. Today, one would say that if (\ref{e3})
is true, then so must (\ref{e2}) because \textquotedblleft everything that's
true for abelian varieties is true for motives\textquotedblright%
.\footnote{This reasoning is circular: what we hope to be true for motives is
partly based on our hope that the Tate and Hodge conjectures are true.}
However, when Tate was thinking about these things, motives didn't exist.
Apparently, the first text in which the notion of a \textit{motif }appears is
Grothendieck's letter to Serre of August 16, 1964 (\cite{grothendieckS}, p173, p276).
\end{aside}

\begin{aside}
\label{t9b}Theorem \ref{t7} has been extended to abelian varieties over fields
finitely generated over the prime field by Zarhin and Faltings
(\cite{zarhin1974i, zarhin1974a, faltings1983, faltingsW1984}).
\end{aside}

\subsection{Abelian varieties with no exotic Tate classes}

\begin{theorem}
\label{t10}The Tate conjecture $T(A)$ holds for any abelian variety $A$ such
that, for some $\ell$, the $\mathbb{Q}{}_{\ell}$-algebra $\mathcal{T}{}_{\ell
}^{\ast}(A)$ is generated by $\mathcal{T}{}_{\ell}^{1}(A)$; in fact, the
equivalent statements of (\ref{t2}) hold for such an $A$ and all $r\geq0$.
\end{theorem}

\begin{proof}
If the $\mathbb{Q}{}_{\ell}$-algebra $\mathcal{T}{}_{\ell}^{\ast}(A)$ is
generated by $\mathcal{T}{}_{\ell}^{1}(A)$, then, because the latter is
spanned by algebraic classes (\ref{t9}), so is the former. Thus $T^{r}%
(A,\ell)$ holds for all $r$, and as $S(A)$ is known, this implies that the
equivalent statements of (\ref{t2}) hold for $A$.
\end{proof}

For an abelian variety $A$, let $\mathcal{L}_{\ell}^{\ast}(A)$ be the
$\mathbb{Q}{}$-subalgebra of $H^{2\ast}(A,\mathbb{Q}{}_{\ell}(\ast))$
generated by the divisor classes. Then $\mathcal{L}_{\ell}^{\ast}%
(A)\cdot\mathbb{Q}{}_{\ell}\subset\mathcal{T}{}_{\ell}^{\ast}(A)$. The
elements of $\mathcal{L}_{\ell}^{\ast}(A)$ are called the \emph{Lefschetz
classes }on $A$, and the Tate classes not in $\mathcal{L}{}_{\ell}^{\ast}(A)$
are said to be \emph{exotic}.

An abelian variety $A$ is said to have \emph{sufficiently many endomorphisms
}if $\End^{0}(A)$ contains an \'{e}tale $\mathbb{Q}{}$-subalgebra of degree
$2\dim A$ over $\mathbb{Q}{}$. Tate's theorem (\ref{t7}) implies that every
abelian variety over $\mathbb{F}{}$ has sufficiently many endomorphisms.

Let $A$ be an abelian variety with sufficiently many endomorphisms over an
algebraically closed field $k$, and let $C(A)$ be the centre of $\End^{0}(A)$.
The Rosati involution $^{\dagger}$ defined by any polarization of $A$
stabilizes $C(A)$, and its restriction to $C(A)$ is independent of the choice
of the polarization. Define $L(A)$ to be the algebraic group over
$\mathbb{Q}{}$ such that, for any commutative $\mathbb{Q}{}$-algebra $R$,%
\[
L(A)(R)=\{a\in C(A)\otimes R\mid aa^{\dagger}\in R^{\times}\}.
\]
It is a group of multiplicative type (not necessarily connected), which acts
in a natural way on the cohomology groups $H^{2\ast}(A^{n},\mathbb{Q}{}_{\ell
}(\ast)){}$ for all $n$. It is known that the subspace fixed by $L(A)$
consists of the Lefschetz classes,%
\begin{equation}
H^{2\ast}(A^{n},\mathbb{Q}{}_{\ell}(\ast))^{L(A)}=\mathcal{L}{}_{\ell}^{\ast
}(A^{n})\cdot\mathbb{Q}{}_{\ell}\text{, all }n\text{ and }\ell, \label{e4}%
\end{equation}
(see \cite{milne1999lc}). Let $\pi$ be a Frobenius endomorphism of $A$. Some
power $\pi^{N}$ of $\pi$ lies in $C(A)$, hence in $L(A)(\mathbb{Q}{})$, and
(\ref{e4}) shows that no power of $A$ has an exotic Tate class if and only if
$\pi^{N}$ is Zariski dense in $L(A)$. This gives the following explicit criterion:

\bquote

\begin{plain}
\label{t10c}Let $A$ be a simple abelian variety over $\mathbb{F}{}$, let
$\pi_{A}$ be a $q$-Frobenius endomorphism of $A$ lying in $C(A)$, and let
$(\pi_{i})_{1\leq i\leq2s}$ be the roots in $\mathbb{C}{}$ of the minimum
polynomial of $\pi_{A}$, numbered so that $\pi_{i}\pi_{i+s}=q$. Then no power
of $A$ has an exotic Tate class (and so the Tate conjecture holds for all
powers of $A$) if and only if $\{\pi_{1},\ldots,\pi_{s},q\}$ is a
$\mathbb{Z}{}$-linearly independent in $\mathbb{C}{}^{\times}$ (i.e., $\pi
_{1}^{m_{1}}\cdots\pi_{s}^{m_{s}}=q^{m}$, $m_{i},m\in\mathbb{Z}$, implies
$m_{1}=\cdots=m_{s}=0=m$).
\end{plain}

\equote\noindent Kowalski (2005, 2.2(2), 2.7)\nocite{kowalski2005} verifies
the criterion for any simple ordinary abelian variety $A$ such that the Galois
group of $\mathbb{Q}{}[\pi_{1},\ldots,\pi_{2g}]$ over $\mathbb{Q}{}$ is the
full group of permutations of $\{\pi_{1},\ldots,\pi_{2g}\}$ preserving the
relations $\pi_{i}\pi_{i+g}=q$.

More generally, let $P(A)$ be the smallest algebraic subgroup of $L(A)$
containing a Frobenius element. Then no power of $A$ has an exotic Tate class
if and only if $P(A)=L(A)$. Spiess (1999)\nocite{spiess1999} proves this for
products of elliptic curves, and Zarhin (1991)\nocite{zarhin1991} and Lenstra
and Zarhin (1993)\nocite{lenstraZ1993} prove it for certain abelian varieties.
See also \cite{milne2001}, A7.

\begin{remark}
\label{t10d}Let $K$ be a CM subfield of $\mathbb{C}{}$, finite and Galois over
$\mathbb{Q}{}$, and let $G=\Gal(K/\mathbb{Q}{})$. We say that an abelian
variety $A$ is \emph{split} by $K$ if $\End^{0}(A)$ is \emph{split} by $K$,
i.e., $\End^{0}(A)\otimes_{\mathbb{Q}{}}K$ is isomorphic to a product of
matrix algebras over $\mathbb{Q}{}$.

Let $A$ be a simple abelian variety over $\mathbb{F}{}$, and let $\pi$ be a
$q$-Frobenius endomorphism for $A$. If $A$ is split by $K$, then, for any
$p$-adic prime $w$ of $K$
\[
f_{A}(w)\overset{\text{{\tiny def}}}{=}\frac{\ord_{w}(\pi)}{\ord_{w}(q)}%
[K_{w}\colon\mathbb{Q}{}_{p}]
\]
lies in $\mathbb{Z}{}$ (apply \cite{tate1966e}, p142). Clearly, $f_{A}(w)$ is
independent of the choice of $\pi$, and the equality $\pi\cdot\iota\pi=q$
implies that $f_{A}(w)+f_{A}(\iota w)=$ $[K_{w}\colon\mathbb{Q}{}_{p}]$. Let
$W$ be the set of $p$-adic primes of $K$, and let $d$ be the local degree
$[K_{w}\colon\mathbb{Q}{}_{p}]$. The map $A\mapsto f_{A}$ defines a bijection
from the set of isogeny classes of simple abelian varieties over $\mathbb{F}%
{}$ split by $K$ to the set%
\begin{equation}
\{f\colon W\rightarrow\mathbb{Z}{}\mid\text{ }f+\iota f=d,\quad\text{ }0\leq
f(w)\leq d\text{ all }w\}.\label{e5}%
\end{equation}
The character groups of $P(A)$ and $L(A)$ have a simple description in terms
of $f_{A}$, and so computing the dimensions of $P(A)$ and $L(A)$ is only an
exercise in linear algebra (cf. \cite{wei1993}, Part I; \cite{milne2001}, A7).
\end{remark}

\subsection{Abelian varieties with exotic Tate classes}

Typically, some power of a simple abelian variety over $\mathbb{F}$ will have
exotic Tate classes.

\begin{proposition}
\label{t11}Let $K$ be a CM-subfield of $\mathbb{C}{}$, finite and Galois over
$\mathbb{Q}{}$, with Galois group $G$.

\begin{enumerate}
\item There exists a CM-field $K_{0}$ such that, if $K\supset K_{0}$, then 

\begin{enumerate}
\item the decomposition groups $D_{w}$ in $G$ of the $p$-adic primes $w$ of
$K$ are not normal in $G$,

\item $\iota$ acts without fixed points on the set $W$ of $p$-adic primes of
$K$.
\end{enumerate}

\item Assume $K\supset K_{0}$, so that $D_{w}\neq D_{w^{\prime}}$ for some
$w,w^{\prime}\in W$. Let $A$ be the simple abelian variety over $\mathbb{F}{}$
corresponding (as in \ref{t10d}) to a function $f\colon W\rightarrow
\mathbb{Z}{}$ such that $f(w)\neq f(w^{\prime})$ and neither $f(w)$ nor
$f(w^{\prime})$ lie in $f(W\smallsetminus\{w,w^{\prime}\})$. Then some power
of $A$ supports an exotic Tate class.
\end{enumerate}
\end{proposition}

\begin{proof}
(a) There exists a totally real field $F$ with Galois group $S_{5}$ over
$\mathbb{Q}{}$ having at least three $p$-adic primes (\cite{wei1993}, 1.6.9),
we can take $K_{0}=F\cdot Q$ where $Q$ is any quadratic imaginary field in
which $p$ splits.

(b) We have%
\[
\dim P(A)\leq t+1,\quad\text{where }|W|=2t,
\]
(\cite{wei1993}, 1.4.4). Let
\[
H=\{g\in G\mid f(gw)=f(w)\text{ for all }w\in W\}.
\]
Then $C(A)\approx K^{H}$, and so%
\[
\dim L(A)=\tfrac{1}{2}(G\colon H)+1.
\]
The conditions on $f$ imply that $H\subset D_{w}\cap D_{w^{\prime}}$, which is
properly contained in $D_{w}$. As $2t=(G\colon D_{w})$, we see that $\dim
L(A)>\dim P(A)$, and so $L(A)\neq P(A)$.
\end{proof}

\noindent The set (\ref{e5}) has $t^{d}$ elements, of which $t(t-1)(t-2)^{d-2}%
$ satisfy the conditions of (\ref{t11}b). As we let $K$ grow, the ratio
$t(t-1)(t-2)^{d-2}/t^{d}\ $tends to $1$, which justifies the statement
preceding the proposition$.$

\begin{theorem}
\label{t12}There exists a family of abelian varieties $A$ over $\mathbb{F}{}$
for which the Tate conjecture $T(A)$ is true and $\mathcal{T}{}_{\ell}^{\ast
}(A)$ is not generated by $\mathcal{T}{}_{\ell}^{1}(A)$.
\end{theorem}

\begin{proof}
See \cite{milne2001} (the proof makes use of \cite{schoen1988, schoen1998}).
\end{proof}

\section{$K3$ surfaces}

The next theorem was proved by Artin and Swinnerton-Dyer
(1973).\nocite{artinS1973}

\begin{theorem}
\label{t13}The Tate conjecture holds for $K3$ surfaces over $\mathbb{F}{}$
that admit a pencil of elliptic curves.
\end{theorem}

\begin{pf}
Let $X$ be an elliptic $K3$ surface, and let $f\colon X\rightarrow\mathbb{P}%
{}^{1}$ be the pencil of elliptic curves. A transcendental Tate class on $X$
gives rise to a sequence $(p_{n})_{n\geq1}$ of elements of the
Tate-Shafarevich group of the generic fibre of $E=X_{\eta}$ of $f$ such that
$\ell p_{n+1}=p_{n}$ for all $n$. From these elements, we get a tower%
\[
\cdots\rightarrow P_{n+1}\rightarrow P_{n}\rightarrow\cdots
\]
of principle homogeneous spaces for $E$ over $\mathbb{F}{}(\mathbb{P}{}^{1})$.
By studying the behaviour of certain invariants attached to the $P_{n}$, Artin
and Swinnerton-Dyer were able to show that no such tower can exist.
\end{pf}

\noindent For later work on K3 surfaces, see \cite{nygaard1983, nygaardO1985,
zarhin1996}.

\begin{aside}
\label{t14}With the proof of the theorems of Tate (\ref{t7}) and of Artin and
Swinnerton-Dyer (\ref{t13}), there was considerable optimism in the early
1970s that the Tate conjecture would soon be proved for surfaces over finite
fields --- all one had to do was attach a sequence of algebro-geometric
objects to a transcendental Tate class, and then prove that the sequence
couldn't exist. However, the progress since then has been meagre. For example,
we still don't know the Tate conjecture for all $K3$ surfaces over
$\mathbb{F}{}$.
\end{aside}

\section{Algebraic classes have the Tannaka property}

Let $\mathcal{S}{}$ be a class of algebraic varieties over $\mathbb{F}{}$
containing the projective spaces and closed under disjoint unions and products
and passage to a connected component.

\begin{theorem}
\label{t15}Let $H_{W}$ be a Weil cohomology theory on the algebraic varieties
over $\mathbb{F}{}$ with coefficients in a field $Q$ in the sense of
\cite{kleiman1994}, \S 3). Assume that for all $X\in\mathcal{S}{}$ the kernel
of the cycle map $Z^{\ast}(X)\rightarrow H_{W}^{2\ast}(X)(\ast)$ consists
exactly of the cycles numerically equivalent to zero. Let $X\in\mathcal{S}{}$,
and let $G_{X}$ be the algebraic subgroup of $\GL(H_{W}^{\ast}(X))\times
\GL(Q{}(1))$ fixing all algebraic classes on all powers of $X$. Then the
$Q$-vector space $H_{W}^{2\ast}(X^{n})(\ast)^{G_{X}}$ is spanned by algebraic
classes for all $n$.
\end{theorem}

\begin{proof}
Let $\Mot(\mathbb{F}{})$ be the category of motives over $\mathbb{F}{}$ based
on the varieties in $\mathcal{S}{}$ and using numerical equivalence classes of
algebraic cycles as correspondences. Because the K\"{u}nneth components of the
diagonal are known to be algebraic, $\Mot(\mathbb{F}{})$ is a semisimple
Tannakian category (\cite{jannsen1992}). Our assumption on the cycle map
implies that $H_{W}$ defines a fibre functor $\omega$ on $\Mot(\mathbb{F}{})$.
It follows from the definition of $\Mot(\mathbb{F}{})$, that for any variety
$X$ over $\mathbb{F}{}$ and $n\geq0$,%
\[
Z_{\text{num}}^{\ast}(X^{n})_{\mathbb{Q}{}}\simeq\Hom(\1,h^{2\ast}(X^{n}%
)(\ast))
\]
where $Z_{\text{num}}^{\ast}(X^{n})$ is the graded $\mathbb{Z}$-algebra of
algebraic cycles modulo numerical equivalence and the subscript means that we
have tensored with $\mathbb{Q}{}$. On applying $\omega$ to this isomorphism,
we obtain an isomorphism%
\[
Z_{\text{num}}^{\ast}(X^{n})_{Q}\simeq\Hom_{G}(Q,H_{W}^{2\ast}(X^{n}%
)(\ast))=H_{W}^{2\ast}(X^{n})(\ast)^{G}%
\]
where $G=\underline{\Aut}^{\otimes}(\omega)$. Since $G_{X}$ is the image of
$G$ in $\GL(H_{W}^{\ast}(X))\times\GL(Q{}(1))$, this implies the assertion.
\end{proof}

This is a powerful result: once we know that some cohomology classes are
algebraic, it allows us to deduce that many more are (the group fixing the
classes we \textit{know} to be algebraic contains the group fixing
\textit{all} algebraic classes, and so any class that it fixes is in the span
of the algebraic classes). On applying the theorem to the smallest class
$\mathcal{S}{}$ satisfying the conditions and containing a variety $X$, we
obtain the following criterion:

\bquote

\begin{plain}
\label{t16}Let $X$ be an algebraic variety over $\mathbb{F}{}$ such that
$E(X^{n},\ell)$ holds for all $n$. In order to prove that $T(X^{n},\ell)$
holds for all $n$, it suffices to find enough algebraic classes on the powers
of $X$ for some Frobenius map to be Zariski dense in the algebraic subgroup of
$\GL(H^{\ast}(X,\mathbb{Q}{}_{\ell}))\times\GL(\mathbb{Q}{}_{\ell}(1))$ fixing
the classes.
\end{plain}

\equote\noindent

\begin{aside}
\label{t17}Theorem \ref{t15} holds also for almost-algebraic classes in
characteristic zero in the sense of \cite{serre1974}, 5.2 and \cite{tate1994}, p76.
\end{aside}

\section{On the equality of equivalence relations}

Recall that an abelian variety $A$ has sufficiently many endomorphisms if
$\End^{0}(A)$ contains an \'{e}tale $\mathbb{Q}{}$-subalgebra $E$ of degree
$2\dim A$. It is known that such an $E$ can be chosen to be a product of
CM-fields. There then exists a unique involution $\iota_{E}$ of $E$ such that
$\rho\circ\iota_{E}=\iota\circ\rho$ for any homomorphism $\rho\colon
E\rightarrow\mathbb{C}{}$.

The next theorem (and its proof) is an abstract version of the main theorem of
\cite{clozel1999}.

\begin{theorem}
\label{t18}Let $k$ be an algebraically closed field, and let $H_{W}^{\ast}$ be
a Weil cohomology theory with coefficient field $Q$ . Let $A$ be an abelian
variety over $k$ with sufficiently many endomorphisms, and choose a
$\mathbb{Q}{}$-subalgebra $E$ of $\End^{0}(A)$ as above. Assume that $Q$
splits $E$ (i.e., $E\otimes_{\mathbb{Q}{}}Q\approx Q^{[E\colon\mathbb{Q}{}]}$)
and that there exists an involution $\iota_{Q}$ of $Q$ such that

\begin{itemize}
\item $\sigma\circ\iota_{E}=\iota_{Q}\circ\sigma$ for all $\sigma\colon
E\rightarrow Q$,

\item there exists a Weil cohomology theory $H_{W_{0}}^{\ast}\subset
H_{W}^{\ast}$ with coefficient field $Q_{0}\overset{\text{{\tiny def}}}%
{=}Q^{\langle\iota_{Q}\rangle}$ such that $H_{W}^{\ast}=Q\otimes_{Q_{0}%
}H_{W_{0}}^{\ast}{}$.
\end{itemize}

\noindent Then the kernel of the cycle class map $Z^{\ast}(X)\rightarrow
H_{W}^{2\ast}(X)(\ast)$ consists exactly of the cycles numerically equivalent
to zero.
\end{theorem}

\begin{pf}
Choose a subset $\Phi$ of $\Hom(E,Q\mathbb{)}$ such that
\[
\Hom(E,Q{})=\Phi\sqcup\iota_{Q}\Phi
\]
where
\[
\iota_{Q}\Phi=\{\iota_{Q}\circ\varphi\mid\varphi\in\Phi\}=\Phi\iota_{E}.
\]
The space $H_{W}^{1}(A)$ is free of rank one over $E\otimes_{\mathbb{Q}{}}Q$,
and so $H_{W}^{1}(A)=\bigoplus\nolimits_{\sigma\in\Phi\sqcup\iota\Phi}%
H_{W}^{1}(A)_{\sigma}$ where $H_{W}^{1}(A)_{\sigma}$ is the one-dimensional
$Q$-subspace on which $E$ acts through $\sigma$. Similarly,
\[
H_{W}^{r}(A)\simeq\bigwedge\nolimits_{Q}^{r}H_{W}^{1}(A)=\bigoplus
_{I,J,|I|+|J|=r}H_{W}^{r}(A)_{I,J}%
\]
where $I$ and $J$ are subsets of $\Phi$ and $\iota\Phi$ respectively, and
$H_{W}^{r}(A)_{I,J}$ is the one-dimensional subspace on which $a\in E$ acts as
$\prod\nolimits_{\sigma\in I\sqcup J}\sigma a$. Let $\mathcal{A}{}_{W}^{\ast
}(A)$ be the $\mathbb{Q}{}$-subalgebra of $H_{W}^{2\ast}(A)(\ast)$ generated
by the algebraic classes. Since $\mathcal{A}{}_{W}^{\ast}(A)\subset H_{W_{0}%
}^{2\ast}(A)(\ast)$, the $Q$-subspace $Q\cdot\mathcal{A}{}_{W}^{\ast}(A)$ of
$H_{W}^{2\ast}(A)(\ast)$ is stable under the involution $\iota_{H}$ of
$H_{W}^{2\ast}(A)(\ast)$ defined by the $Q_{0}$-substructure $H_{W_{0}}%
^{2\ast}(A)(\ast)$. Therefore, if $H_{W}^{2r}(A)_{I\sqcup J}(r)$ is contained
in $Q\cdot\mathcal{A}{}_{W}^{\ast}(A)$, then so also is
\[
\iota_{W}H_{W}^{2r}(A)_{I\sqcup J}(r)=H_{W}^{2r}(A)_{\iota_{Q}I\sqcup\iota
_{Q}J}(r).
\]
Using this, Clozel constructs, for every nonzero algebraic class, an algebraic
class of complementary degree whose product with the first class is nonzero.
\end{pf}

\begin{corollary}
\label{t19}For an abelian variety $A$ over $\mathbb{F}{}$, there is an
infinite set of primes $\ell\neq p$ such that $E(A^{n},\ell)$ is true for all
$n$.
\end{corollary}

\begin{proof}
Choose a $\mathbb{Q}{}$-subalgebra $E$ of $\End^{0}(A)$ as before, and let
$Q^{\prime}$ be the composite of the images of $E$ in $\mathbb{C}{}$ under
homomorphisms $E\rightarrow\mathbb{C}{}$. Then $Q^{\prime}$ is a finite Galois
extension of $\mathbb{Q}{}$ that splits $E$ and it is a CM-field. Let $S$ be
the set of primes $\ell\neq p$ such that $\iota$ lies in the decomposition
group of some $\ell$-adic prime $v$ of $Q^{\prime}$. For example, if $\iota$
is the Frobenius element of an $\ell$-adic prime of $Q^{\prime}$, then
$\ell\in S$, and so $S$ has density $>0$. The Weil cohomology theory
$H_{W}\overset{\text{{\tiny def}}}{=}H_{\ell}\otimes Q_{v}^{\prime}$ satisfies
the hypotheses of the theorem for $A$. For $A^{n}$, we can choose the
$\mathbb{Q}{}$-algebra to be $E^{n}$ acting diagonally and use the same set
$S$.
\end{proof}

On combining (\ref{t19}) with (\ref{t16}) we obtain the following criterion:

\bquote

\begin{plain}
\label{t19a}Let $A$ be an abelian variety over $\mathbb{F}$. In order to prove
that $T(A^{n})$ holds for all $n$, it suffices to find enough algebraic
classes on powers of $A$ for some Frobenius endomorphism to be Zariski dense
in the algebraic subgroup of $\GL(H_{\ell}^{\ast}(X))\times\GL(\mathbb{Q}%
{}_{\ell}(1))$ fixing the classes for a suitable $\ell$.
\end{plain}

\equote

\section{The Hodge conjecture and the Tate conjecture}

To go further, we shall need to consider the Hodge conjecture (following
\cite{deligne1982}).

For a variety $X$ over an algebraically closed field $k$ of characteristic
zero, define%
\[
H_{\mathbb{A}{}}^{\ast}(X)=H_{\mathbb{A}{}_{f}}^{\ast}(X)\times H_{\mathrm{dR}%
}^{\ast}(X)\text{ where }H_{\mathbb{A}{}_{f}}^{\ast}(X)=\left(  \varprojlim
_{m}H^{\ast}(X_{\mathrm{et}},\mathbb{Z}{}/m\mathbb{Z}{})\right)
\otimes_{\mathbb{Z}{}}\mathbb{Q}{}.
\]
For any algebraically closed field $K$ containing $k,$%
\[
H_{\mathbb{A}{}_{f}}^{\ast}(X_{K})\simeq H_{\mathbb{A}{}_{f}}^{\ast}(X)\text{
and }H_{\mathrm{dR}}^{\ast}(X_{K})\simeq H_{\mathrm{dR}}^{\ast}(X)\otimes
_{k}K,
\]
and so there is a canonical homomorphism $H_{\mathbb{A}{}}^{\ast
}(X)\rightarrow H_{\mathbb{A}{}}^{\ast}(X_{K})$.

Let $\sigma$ be a homomorphism $k\rightarrow\mathbb{C}{}$%
.\footnote{Throughout, I assume that $k$ is not too big to be embedded into
$\mathbb{C}{}$. For fields that are \textquotedblleft too
big\textquotedblright, one can use property (c) of $\mathcal{B}{}%
_{\mathrm{abs}}^{\ast}$ below as a definition of $\mathcal{B}{}_{\mathrm{abs}%
}^{\ast}(K)$.} An element of $H_{\mathbb{A}{}}^{2\ast}(X)(\ast)$ is
\emph{Hodge relative to }$\sigma$ if its image in $H_{\mathbb{A}{}}^{2\ast
}(X_{\mathbb{C}{}})(\ast)$ is a Hodge class, i.e., lies in $H^{2\ast
}(X_{\mathbb{C}{}},\mathbb{Q}{})(\ast)\subset H_{\mathbb{A}{}}^{2\ast
}(X_{\mathbb{C}{}})(\ast)$ and is of type $(0,0)$.

\begin{conjecture}
[Deligne]\label{t20}If an element of $H_{\mathbb{A}{}}^{2\ast}(X)(\ast)$ is
Hodge relative to one homomorphism $\sigma\colon k\rightarrow\mathbb{C}{}$,
then it is Hodge relative to every such homomorphism.
\end{conjecture}

An element of $H_{\mathbb{A}{}}^{2\ast}(X)(\ast)$ is \emph{absolutely Hodge}
if it is Hodge relative to every $\sigma$. Let $\mathcal{B}_{\mathrm{abs}%
}^{\ast}(X)$ be the set of absolutely Hodge classes on $X$. Then
$\mathcal{B}_{\mathrm{abs}}^{\ast}(X)$ is a graded $\mathbb{Q}{}$-subalgebra
of $H_{\mathbb{A}}^{2\ast}(X)(\ast)$, and Deligne (1982)
shows:\nocite{deligne1982}

\begin{enumerate}
\item for every regular map $f\colon X\rightarrow Y$, $f^{\ast}$ maps
$\mathcal{B}_{\mathrm{abs}}^{\ast}(Y)$ into $\mathcal{B}_{\mathrm{abs}}^{\ast
}(X)$ and $f_{\ast}$ maps $\mathcal{B}_{\mathrm{abs}}^{\ast}(X)$ into
$\mathcal{B}_{\mathrm{abs}}^{\ast}(Y)$;

\item for every $X$, $\mathcal{B}_{\mathrm{abs}}^{\ast}(X)$ contains the
algebraic classes;

\item for every homomorphism $k\rightarrow K$ of algebraically closed fields,
$\mathcal{B}_{\mathrm{abs}}^{\ast}(X)\simeq\mathcal{B}_{\mathrm{abs}}^{\ast
}(X_{K})$;

\item for any model $X_{1}$ of $X$ over a subfield $k_{1}$ of $k$ with $k$
algebraic over $k_{1}$, $\Gal(k/k_{1})$ acts on $\mathcal{B}_{\mathrm{abs}%
}^{\ast}(X)$ through a finite quotient.
\end{enumerate}

\noindent Property (d) shows that the image of $\mathcal{B}_{\mathrm{abs}%
}^{\ast}(X)$ in $H^{2\ast}(X,\mathbb{Q}{}_{\ell}(\ast))$ consists of Tate classes.

\begin{theorem}
\label{t21}If the Tate conjecture holds for $X$, then all absolutely Hodge
classes on $X$ are algebraic.
\end{theorem}

\begin{proof}
Let $\mathcal{A}{}^{\ast}(X)$ be the $\mathbb{Q}{}$-subspace of $H_{\mathbb{A}%
{}}^{2\ast}(X)(\ast)$ spanned by the classes of algebraic cycles, and consider
the diagram defined by a homomorphism $k\rightarrow\mathbb{C}{}$,
\[
\begin{CD}
@.
\mathcal{B}^{\ast}(X_{\mathbb{C}}) @>{\subset}>>
H_{B}^{2\ast}(X_{\mathbb{C}})(\ast) @>{\subset}>>
H^{2\ast}_{\ell}(X_{\mathbb{C}})(\ast)\\
@.@AA{\bigcup}A@.@AA{\simeq}A\\
\mathcal{A}^{\ast}(X)@>{\subset}>>
\mathcal{B}_{\mathrm{abs}}^{\ast}(X) @>>>
\mathcal{T}^{\ast}_{\ell}(X) @>{\subset}>>
H^{2\ast}_{\ell}(X)(\ast).
\end{CD}
\]
The four groups at upper left are finite dimensional $\mathbb{Q}{}$-vector
spaces, and the map at top right gives an isomorphism $H_{B}^{2\ast
}(X_{\mathbb{C}{}})(\ast)\otimes_{\mathbb{Q}{}}\mathbb{Q}{}_{\ell}\simeq
H_{\ell}^{2\ast}(X_{\mathbb{C}{}})(\ast)$. Therefore, on tensoring the
$\mathbb{Q}{}$-vector spaces in the above diagram with $\mathbb{Q}{}_{\ell}$,
we get injective maps%
\[
\mathcal{A}{}^{\ast}(X)\otimes_{\mathbb{Q}{}}\mathbb{Q}{}_{\ell}%
\hookrightarrow\mathcal{B}_{\mathrm{abs}}^{\ast}(X)\otimes_{\mathbb{Q}{}%
}\mathbb{Q}{}_{\ell}\hookrightarrow\mathcal{T}{}_{\ell}^{\ast}(X).
\]
If the Tate conjecture holds for $X$, then the composite of these maps is an
isomorphism, and so the first is also an isomorphism. This implies that
$\mathcal{A}{}^{\ast}(X)=\mathcal{B}_{\mathrm{abs}}^{\ast}(X)$.
\end{proof}

\begin{theorem}
\label{t22}For varieties $X$ satisfying Deligne's conjecture, the Tate
conjecture for $X$ implies the Hodge conjecture for $X_{\mathbb{C}{}}$.
\end{theorem}

\begin{proof}
For any homomorphism $k\rightarrow\mathbb{C}{}$, the homomorphism
$H_{\mathbb{A}{}}^{2\ast}(X)(\ast)\hookrightarrow H_{\mathbb{A}{}}^{2\ast
}(X_{\mathbb{C}{}})(\ast)$ maps $\mathcal{B}_{\mathrm{abs}}^{\ast}(X)$ into
$\mathcal{B}{}^{\ast}(X_{\mathbb{C}{}})$. When Deligne's conjecture holds for
$X$, $\mathcal{B}_{\mathrm{abs}}^{\ast}(X)\simeq\mathcal{B}{}^{\ast
}(X_{\mathbb{C}{}})$. Therefore, if $\mathcal{B}_{\mathrm{abs}}^{\ast}(X)$
consists of algebraic classes, so also does $\mathcal{B}{}^{\ast
}(X_{\mathbb{C}{}})$.
\end{proof}

\begin{aside}
\label{t23}The Hodge conjecture is known for divisors, and the Tate conjecture
is generally expected to be true for divisors. However, there is little
evidence for either conjecture in higher codimensions, and hence little reason
to believe them. On the other hand, Deligne believes his conjecture to be
true.\footnote{When asked at the workshop, Tate said that he believes the Tate
conjecture for divisors, but that in higher codimension he has no idea. Also
it worth recalling that Hodge didn't conjecture the Hodge conjecture: he
raised it as a problem\ (\cite{hodge1952}, p184). There is considerable
evidence (and some proofs) that the classes predicted to be algebraic by the
Hodge and Tate conjectures do behave as if they are algebraic, at least in
some respects, but there is little evidence that they are actually algebraic.}
\end{aside}

\begin{aside}
\label{t23a}As Tate pointed out at the workshop, one reason the Tate
conjecture is harder than the Hodge conjecture is that it doesn't tell you
which cohomology classes are algebraic; it only tells you the $\mathbb{Q}%
{}_{\ell}$-span of the algebraic classes.
\end{aside}

\subsection{Deligne's theorem on abelian varieties}

The following is an abstract version of the main theorem of \cite{deligne1982}.

\begin{theorem}
\label{t24}Let $k$ be an algebraically closed subfield of $\mathbb{C}{}$.
Suppose that for every abelian variety $A$ over $k{}$, we have a graded
$\mathbb{Q}{}$-subalgebra $\mathcal{C}{}^{\ast}(A)$ of $\mathcal{B}{}^{\ast
}(A_{\mathbb{C}{}})$ such that

\textnf{(A1)} for every regular map $f\colon A\rightarrow B$ of abelian
varieties over $k$, $f^{\ast}$ maps $\mathcal{C}{}^{\ast}(B)$ into
$\mathcal{C}{}^{\ast}(A)$ and $f_{\ast}$ maps $\mathcal{C}{}^{\ast}(A)$ into
$\mathcal{C}{}^{\ast}(B)$;

\textnf{(A2)} for every abelian variety $A$, $\mathcal{C}{}^{1}(A)$ contains
the divisor classes; and

\textnf{(A3)} let $f\colon\mathcal{A}{}\rightarrow S$ be an abelian scheme
over a connected smooth (not necessarily complete) $k$-variety $S$, and let
$\gamma\in\Gamma(S_{\mathbb{C}{}},R^{2\ast}f_{\mathbb{C}{}\ast}\mathbb{Q}%
{}(\ast)$); if $\gamma_{t}$ is a Hodge class for all $t\in S(\mathbb{C}{})$
and $\gamma_{s}$ lies in $\mathcal{C}^{\ast}{}(A_{s})$ for one $s\in S(k)$,
then it lies in $\mathcal{C}^{\ast}(A_{s})$ for all $s\in S(k)$.

\noindent\smallskip\noindent Then $\mathcal{C}{}^{\ast}(A)\simeq\mathcal{B}%
{}^{\ast}(A_{\mathbb{C}{}})$ for all abelian varieties over $k$.
\end{theorem}

\noindent For the proof, see the endnotes to \cite{deligne1982}). I list three
applications of this theorem.

\begin{theorem}
\label{t25}In order to prove the Hodge conjecture for abelian varieties, it
suffices to prove the variational Hodge conjecture.
\end{theorem}

\begin{proof}
Take $\mathcal{C}{}^{\ast}(A)$ to be the $\mathbb{Q}{}$-subspace of
$H_{B}^{2\ast}(A_{\mathbb{C}{}})(\ast)$ spanned by the classes of algebraic
cycles on $A$. Clearly (A1) and (A2) hold, and (A3) is (one form of) the
variational Hodge conjecture.
\end{proof}

The following is the original version of the main theorem of
\cite{deligne1982}.

\begin{theorem}
\label{t26}Deligne's conjecture holds for all abelian varieties $A$ over $k$
(hence the Tate conjecture implies the Hodge conjecture for abelian varieties).
\end{theorem}

\begin{proof}
Take $\mathcal{C}{}^{\ast}(A)$ to be $\mathcal{B}_{\mathrm{abs}}^{\ast}(A)$.
Clearly (A2) holds, and we have already noted that (A1) holds. That (A3) holds
is proved in Deligne 1982.
\end{proof}

The theorem implies that, for an abelian variety $A$ over an algebraically
closed field $k$ of characteristic zero, any homomorphism $k\rightarrow
\mathbb{C}{}$ defines an isomorphism $\mathcal{B}{}_{\text{abs}}^{\ast
}(A)\rightarrow\mathcal{B}{}^{\ast}(A_{\mathbb{C}{}})$. In view of this, I now
write $\mathcal{B}{}^{\ast}(A)$ for $\mathcal{B}{}_{\text{abs}}^{\ast}(A)$ and
call its elements the \emph{Hodge classes} on $A$.

\subsubsection{Motivated classes (following \cite{andre1996})}

Let $k$ be an algebraically closed${}$ field, and let $H_{W}^{\ast}$ be a Weil
cohomology theory on the varieties over $k$ with coefficient field $Q$. For a
variety $X$ over $k$, let $L$ and $\Lambda$ be the operators defined by a
hyperplane section of $X$, and define
\[
\mathcal{\mathcal{E}{}}{}^{\ast}(X)=Q{}[L,\Lambda]\cdot\mathcal{A}{}_{W}%
^{\ast}(X)\subset H_{W}^{2\ast}(X)(\ast)
\]
where $\mathcal{A}{}_{W}^{\ast}(X)$ is the $\mathbb{Q}{}$-subspace of
$H_{W}^{2\ast}(X)(\ast)$ generated by algebraic classes. Then $\mathcal{E}%
{}^{\ast}(X)$ is a graded $Q$-subalgebra of $H_{W}^{2\ast}(X)(\ast)$, but
these subalgebras are not (obviously) stable under direct images. However,
when we define%
\[
\mathcal{C}{}^{\ast}(X)=\bigcup\nolimits_{Y}p_{\ast}\mathcal{\mathcal{E}{}%
}^{\ast}(X\times Y),
\]
then $\mathcal{C}{}^{\ast}(X)$ is a graded $Q$-subalgebra of $H_{W}^{2\ast
}(X)(\ast)$, and these algebras satisfy (A1). They obviously satisfy (A2).

\begin{theorem}
\label{t27}Let $k$ be an algebraically closed subfield of $\mathbb{C}{}$, and
let $H_{W}$ be the Weil cohomology theory $X\mapsto H_{B}^{\ast}%
(X_{\mathbb{C}{}})$. For every abelian variety $A$, $\mathcal{C}{}^{\ast
}(A)=\mathcal{B}{}^{\ast}(A_{\mathbb{C}{}})$.
\end{theorem}

\begin{proof}
Clearly (A2) holds, and that (A3) holds is proved in \cite{andre1996}, 0.5.
\end{proof}

The elements of $\mathcal{C}^{\ast}(X)$ are called \emph{motivated classes}.

\begin{aside}
\label{t23b}As Ramakrishnan pointed out at the workshop, since proving the
Hodge conjecture is worth a million dollars and the Tate conjecture is harder,
it should be worth more.
\end{aside}

\section{Rational Tate classes}

There are by now many papers proving that, if the Tate conjecture is true,
then something else even more wonderful is true. But what if we are never able
to decide whether the Tate conjecture is true? or worse, what if it turns out
to be false? In this section, I suggest an alternative to the Tate conjecture
for varieties over finite fields, which appears to be much more accessible,
and which has some of the same consequences.

An abelian variety with sufficiently many endomorphisms over an algebraically
closed field of characteristic zero will now be called a \emph{CM abelian
variety}. Let $\mathbb{Q}{}^{\mathrm{al}}$ be the algebraic closure of
$\mathbb{Q}{}$ in $\mathbb{C}{}$. Then the functor $A\rightsquigarrow
A_{\mathbb{C}{}}$ from CM abelian varieties over $\mathbb{Q}{}^{\mathrm{al}}$
to CM abelian varieties over $\mathbb{C}{}$ is an equivalence of categories
(see, for example, \cite{milneCM}, \S 7).

Fix a $p$-adic prime $w$ of $\mathbb{Q}{}^{\mathrm{al}}$, and let
$\mathbb{F}{}$ be its residue field. It follows from the theory of N\'{e}ron
models that there is a well-defined reduction functor $A\rightsquigarrow
A_{0}$ from CM abelian varieties over $\mathbb{Q}{}^{\mathrm{al}}$ to abelian
varieties over $\mathbb{F}{}$, which the Honda-Tate theorem shows to be
surjective on isogeny classes.

Let $\mathbb{Q}{}_{w}^{^{\mathrm{al}}}$ be the completion of $\mathbb{Q}{}$ at
$w$. For a variety $X$ over $\mathbb{F}{}$, define%
\[
H_{\mathbb{A}{}}^{\ast}(X)=H_{\mathbb{A}{}_{f}}^{\ast}(X)\times H_{p}^{\ast
}(X)\text{ where }\left\{
\begin{array}
[c]{lll}%
H_{\mathbb{A}{}_{f}}^{\ast}(X) & = & \left(  \varprojlim_{m,p\nmid m}H^{\ast
}(X_{\mathrm{et}},\mathbb{Z}{}/m\mathbb{Z}{})\right)  \otimes_{\mathbb{Z}{}%
}\mathbb{Q}{}\\
H_{p}^{\ast}(X) & = & H_{\mathrm{crys}}^{\ast}(X)\otimes_{W(\mathbb{F}{}%
)}\mathbb{Q}{}_{w}^{^{\mathrm{al}}}.
\end{array}
\right.
\]
For a CM abelian variety $A$ over $\mathbb{Q}{}^{\mathrm{al}}$,%
\begin{align*}
H_{\mathbb{A}{}_{f}}^{\ast}(A_{K})\left(  \text{not-}p\right)   &  \simeq
H_{\mathbb{A}{}_{f}}^{\ast}(A_{0})\\
H_{\mathrm{dR}}^{\ast}(A)\otimes_{\mathbb{Q}{}^{\mathrm{al}}}\mathbb{Q}{}%
_{w}^{^{\mathrm{al}}}  &  \simeq H_{\mathrm{crys}}^{\ast}(A_{0})\otimes
_{W(\mathbb{F}{})}\mathbb{Q}{}_{w}^{^{\mathrm{al}}},
\end{align*}
and so there so there is a canonical (specialization) map $H_{\mathbb{A}{}%
}^{\ast}(A)\rightarrow H_{\mathbb{A}{}}^{\ast}(A_{0})$.

Let $\mathcal{S}{}$ be a class of smooth projective varieties over
$\mathbb{F}{}$ satisfying the following condition:

\begin{quote}
(*) it contains the abelian varieties and projective spaces and is closed
under disjoint unions, products, and passage to a connected component.
\end{quote}

\begin{definition}
\label{t41}A family $(\mathcal{R}{}^{\ast}(X))_{X\in\mathcal{S}{}}$ with each
$\mathcal{R}{}^{\ast}(X)$ a graded $\mathbb{Q}{}$-subalgebra of $H_{\mathbb{A}%
{}}^{2\ast}(X)(\ast)$ is a \emph{good theory of rational Tate classes} if

(R1) for all regular maps $f\colon X\rightarrow Y$ of varieties in
$\mathcal{S}{}$, $f^{\ast}$ maps $\mathcal{R}{}^{\ast}(Y)$ into $\mathcal{R}%
{}^{\ast}(X)$ and $f_{\ast}$ maps $\mathcal{R}{}^{\ast}(X)$ into
$\mathcal{R}{}^{\ast}(Y)$;

(R2) for all varieties $X$ in $\mathcal{S}{}$, $\mathcal{R}{}^{1}(X)$ contains
the divisor classes;

(R3) for all CM abelian varieties $A$ over $\mathbb{Q}{}^{\mathrm{al}}$, the
specialization map $H_{\mathbb{A}{}}^{2\ast}(A)(\ast)\rightarrow
H_{\mathbb{A}{}}^{2\ast}(A_{0})(\ast)$ sends the Hodge classes on $A$ to
elements of $\mathcal{R}{}^{\ast}(A_{0})$;

(R4) for all varieties $X$ in $\mathcal{S}{}$ and all primes $l$ (including
$l=p$), the projection $H_{\mathbb{A}{}}^{2\ast}(X)(\ast)\rightarrow
H_{l}^{2\ast}(X)(\ast)$ defines an isomorphism $\mathcal{R}{}^{\ast}%
(X)\otimes_{\mathbb{Q}{}}\mathbb{Q}{}_{l}\rightarrow\mathcal{T}_{l}^{\ast}(X)$.
\end{definition}

\noindent Thus, (R3) says that there is a diagram%
\[%
\begin{array}
[c]{ccc}%
\mathcal{B}{}^{\ast}(A) & \subset & H_{\mathbb{A}{}}^{2\ast}(A)(\ast)\\
\downarrow &  & \downarrow\\
\mathcal{R}{}^{\ast}(A) & \subset & H_{\mathbb{A}{}}^{2\ast}(A_{0})(\ast),
\end{array}
\]
and (R4) says that $\mathcal{R}{}^{\ast}(X)$ is simultaneously a $\mathbb{Q}%
{}$-structure on each of the $\mathbb{Q}{}_{l}$-spaces $\mathcal{T}{}%
_{l}^{\ast}(X)$ of Tate classes (including for $l=p$). The elements of
$\mathcal{R}{}^{\ast}(X)$ will be called the \emph{rational Tate classes on
}$X$ (for the theory $\mathcal{R}{}$).

The next theorem is an abstract version of the main theorem of
\cite{milne1999lm}.

\begin{theorem}
\label{t42}In the definition of a good theory of rational Tate classes, the
condition \textnf{(R4)} can be weakened to:\bquote\textnf{(R4*)} for all
varieties $X$ in $\mathcal{S}{}$ and all primes $l$, the projection map
$H_{\mathbb{A}{}}^{2\ast}(X)\rightarrow H_{l}^{2\ast}(X)(\ast)$ sends
$\mathcal{R}{}^{\ast}(X)$ into $\mathcal{T}{}_{\ell}^{\ast}(X)$.\equote
\end{theorem}

\noindent In other words, if a family satisfies (R1-3), and (R4*), then it
satisfies (R4). For the proof, see \cite{milne2007rtc}. I list three
applications of it.

Any choice of a basis for a $\mathbb{Q}{}_{l}$-vector space defines a
$\mathbb{Q}{}$-structure on the vector space. Thus, there are many choices of
$\mathbb{Q}{}$-structures on the $\mathbb{Q}{}_{l}$-spaces $\mathcal{T}{}%
_{l}^{\ast}(X)$. The next theorem says, however, that there is
\textit{exactly} one family of choices satisfying the compatibility conditions (R1--4).

\begin{theorem}
\label{t43}There exists at most one good theory of rational Tate classes on
$\mathcal{S}{}$. In other words, if $\mathcal{R}{}_{1}^{\ast}$ and
$\mathcal{R}{}_{2}^{\ast}$ are two such theories, then, for all $X\in
\mathcal{S}{}$, the $\mathbb{Q}{}$-subalgebras $\mathcal{R}{}_{1}^{\ast}(X)$
and $\mathcal{R}{}_{2}^{\ast}(X)$ of $H_{\mathbb{A}{}}^{2\ast}(X)(\ast)$ are equal.
\end{theorem}

\begin{proof}
It follows from (R4) that if $\mathcal{R}{}_{1}^{\ast}$ and $\mathcal{R}{}%
_{2}^{\ast}$ are both good theories of rational Tate classes and
$\mathcal{R}{}_{1}^{\ast}\subset\mathcal{R}{}_{2}^{\ast}$, then they are
equal. But if $\mathcal{R}{}_{1}^{\ast}$ and $\mathcal{R}{}_{2}^{\ast}$
satisfy (R1--4), then $\mathcal{R}_{1}^{\ast}\cap\mathcal{R}{}_{2}^{\ast}$
satisfies (R1--3) and (R4*), and hence also (R4). Therefore $\mathcal{R}{}%
_{1}^{\ast}=\mathcal{R}_{1}^{\ast}\cap\mathcal{R}{}_{2}^{\ast}=\mathcal{R}%
{}_{2}^{\ast}$.
\end{proof}

The following is the main theorem of \cite{milne1999lm}.

\begin{theorem}
\label{t44}The Hodge conjecture for CM abelian varieties implies the Tate
conjecture for abelian varieties over $\mathbb{F}{}$.
\end{theorem}

\begin{proof}
Let $\mathcal{S}{}_{0}$ be the smallest class satisfying (*). For
$X\in\mathcal{S}{}_{0}$, let $\mathcal{R}{}^{\ast}(X)$ be the $\mathbb{Q}{}%
$-subalgebra of $H_{\mathbb{A}{}}^{2\ast}(X)(\ast)$ spanned by the algebraic
classes. The family $(\mathcal{R}{}^{\ast}(X))_{X\in\mathcal{S}{}_{0}}$
satisfies (R1), (R2), and (R4*), and the Hodge conjecture implies that it
satisfies (R3). Therefore it satisfies (R4), which means that the Tate
conjecture holds for abelian varieties over $\mathbb{F}{}$.
\end{proof}

The following is the main theorem of \cite{andre2006b}.

\begin{theorem}
\label{t30}All Tate classes on abelian varieties over $\mathbb{F}{}$ are motivated.
\end{theorem}

\begin{proof}
Let $\mathcal{S}{}_{0}$ be the smallest class satisfying (*). Fix an $\ell$,
and for $X\in\mathcal{S}{}_{0}$, let $\mathcal{C}^{\ast}(X)$ be the
$\mathbb{Q}{}_{\ell}$-algebra of motivated classes in $H_{\ell}^{2\ast
}(X)(\ast)$. The family $(\mathcal{C}{}^{\ast}(X))_{X\in\mathcal{S}{}_{0}}$
satisfies (R1), (R2), and (R4$^{\ast}$) with $\mathbb{A}{}$ replaced by $\ell
$, and Andr\'{e} shows that it satisfies (R3). Therefore, by Theorem \ref{t42}
with $\mathbb{A}{}$ replaced by $\ell$, it satisfies (R4).
\end{proof}

\begin{aside}
\label{t34}Assume that there exists a good theory of rational Tate classes for
abelian varieties over $\mathbb{F}$. Then we expect that all Hodge classes on
all abelian varieties over $\mathbb{Q}{}^{\mathrm{al}}$ with good reduction at
$w$ (not necessarily CM) specialize to rational Tate classes. This will follow
from knowing that every $\mathbb{F}{}$-point on a Shimura variety lifts to a
special point, which is perhaps already known (\cite{zink1983},
\cite{vasiu2003cm}). Note that it implies the \textquotedblleft particularly
interesting\textquotedblright\ corollary of the Hodge conjecture noted in
\cite{deligne2006}, \S 6.
\end{aside}

\section{Reduction to the case of codimension 2}

Throughout this section, $K$ is a CM subfield of $\mathbb{C}{}$, finite and
Galois over $\mathbb{Q}{}$, and $G=\Gal(K/\mathbb{Q}{})$. Recall that an
abelian variety $A$ is said to be split by $K$ if $\End^{0}(A)\otimes
_{\mathbb{Q}{}}K$ is isomorphic to a product of matrix algebras over $K$.

\subsection{CM abelian varieties}

\begin{plain}
\label{t30a}Let $S$ be a finite left $G$-set on which $\iota$ acts without
fixed points, and let $S^{+}$ be a subset of $S$ such that $S=S^{+}\sqcup\iota
S^{+}$. Then $E\overset{\text{{\tiny def}}}{=}\Hom(S,K)^{G}$ is a
$\mathbb{Q}{}$-algebra split by $K$ such that $\Hom(E,K)\simeq S$. The
condition on $\iota$ implies that $E$ is a CM-algebra, and the condition on
$S^{+}$ implies that it is a CM-type on $E$, and so $S^{+}$ defines an
isomorphism $E\otimes_{\mathbb{Q}{}}\mathbb{C}{}\simeq\mathbb{C}{}^{S^{+}}$.
The quotient of $\mathbb{C}{}^{S^{+}}$ by any lattice in $E$ is an abelian
variety of CM-type $(E,S^{+})$. Thus, from such a pair $(S,S^{+})$ we obtain a
CM abelian variety $A(S,S^{+})$, well-defined up to isogeny, which is split by
$K$, and every such abelian variety arises in this way (up to isogeny). For
example, an abelian variety of CM-type $(E,\Phi)$ where $E$ is split by $K$ is
isogenous to $(S,\Phi)$ where $S=\Hom(E,\mathbb{Q}{})$. Note that if
$(S,S^{+})=\bigsqcup(S_{i},S_{i}^{+})$, then, by construction, $A(S,S^{\prime
})$ is isogenous to the product of the varieties $A(S_{i},S_{i}^{+})$.
\end{plain}

\begin{proposition}
\label{t30b}Let $G$ act on the set $S$ of CM-types on $K$ by the rule%
\[
g\Phi=\Phi\circ g^{-1}\overset{\text{{\tiny def}}}{=}\{\varphi\circ g^{-1}%
\mid\varphi\in\Phi\},\quad g\in G,\quad\Phi\in S,
\]
and let $S^{+}$ be the subset of CM-types containing $1_{G}$. The pair
$(S,S^{+})$ defines an abelian variety $A(S,S^{+})$, and every simple CM
abelian variety split by $K$ occurs as an isogeny factor of $A(S,S^{+})$.
\end{proposition}

\begin{proof}
Certainly $S$ is a finite left $G$-set on which $\iota$ acts without fixed
points, and for any CM-type $\Phi$ on $K$, exactly one of $\Phi$ or $\iota
\Phi$ contains $1_{G}$, and so $A(S,S^{\prime})$ is defined.

Let $A$ be a simple CM abelian variety split by $K$, say, of CM-type
$(E,\Phi)$. Fix a homomorphism $i\colon E\rightarrow K$, and let
\[
\Phi^{\prime}=\{g\in G\mid g\circ i\in\Phi\}
\]
--- it is a CM-type on $K$. An element $g$ of $G$ fixes $iE$ if and only if
$g\Phi^{\prime}=\Phi^{\prime}$ (see, for example, \cite{milneCM}, 1.10), and
so $g\circ i\mapsto g\Phi^{\prime}$ is a bijection of $G$-sets from
$\Hom(E,K)$ onto the orbit $O$ of $\Phi^{\prime}$ in $S$. Moreover, $1_{G}\in
g\Phi^{\prime}\overset{\text{{\tiny def}}}{=}\Phi^{\prime}\circ g^{-1}$ if and
only if $g\in\Phi^{\prime}$, i.e., $g\circ i\in\Phi$. Thus, $g\circ i\mapsto
g\Phi^{\prime}$ sends $\Phi$ onto $O\cap S^{+}$. This shows that $A$ is
isogenous to the factor $A(O,O\cap S^{+})$ of $A(S,S^{+})$.
\end{proof}

\subsection{The Hodge conjecture}

Let $2m=[K\colon\mathbb{Q}{}]$. Let $a(2^{m})$ be the hyperplane arrangement
$\{H_{1},\ldots,H_{m}\}$ in $\mathbb{R}{}^{m}$ with $H_{i}$ the coordinate
hyperplane $x_{i}=0$. The hyperplanes $H_{i}$ divide $\mathbb{R}{}%
^{m}\smallsetminus\bigcup H_{i}$ into connected regions%
\[
R(\varepsilon_{1},\ldots,\varepsilon_{m})=\{(x_{1},\ldots,x_{m})\in
\mathbb{R}{}^{m}\smallsetminus\bigcup H_{i}\mid\text{\textrm{sign}}%
(x_{i})=\varepsilon_{i}\}
\]
indexed by the set $\{\pm\}^{m}$. Let $R(2^{m})$ be the set of connected
regions, and let $R(2^{m})^{+}$ be the subset of those with $\varepsilon
_{1}=+$.

Let $\rho$ be a faithful linear representation of $G$ on $\mathbb{R}{}^{m}$
such that $G$ acts transitively on the set $a(2^{m})$ of coordinate
hyperplanes and $\rho(\iota)$ acts as $-1$ on $\mathbb{R}{}^{m}$. The pair
$(R(2^{m}),R(2^{m})^{+})$ then satisfies the conditions of (\ref{t30e}), and
so defines a CM abelian variety $A=A(G,K,\rho)$ of dimension $2^{m-1}$ split
by $K$.

The next statement is the main technical result of \cite{hazama2003}.

\begin{theorem}
\label{t30c}Let $A=A(G,K,\rho)$. For every $n\geq0$, the $\mathbb{Q}{}%
$-algebra $\mathcal{B}{}^{\ast}(A^{n})$ is generated by the classes of degree
$\leq2$.
\end{theorem}

\begin{proof}
See \cite{hazama2003}, \S 7.
\end{proof}

Let $F$ be the largest totally real subfield of $K$. The choice of a CM-type
$\Phi=\{\varphi_{1},\ldots,\varphi_{m}\}$ on $K$ determines a commutative
diagram:%
\[
\begin{CD}
1 @>>> \langle\iota\rangle @>>> G @>>> \Gal(F/\mathbb{Q}{}) @>>> 1\\
@.@VV{\iota\mapsto(-,\ldots,-)}V@VV{\rho_{\Phi}}V @VVV\\
1 @>>> \{\pm\}^{m} @>>> \{\pm\}^{m}\rtimes S_{m} @>>> S_{m} @>>> 1
\end{CD}
\]
Here $\{\pm\}$ denotes the multiplicative group of order $2$, and the
symmetric group $S_{m}$ acts on $\left\{  \pm\right\}  ^{m}$ by permutating
the factors, $\ $%
\[
(\sigma\varepsilon)_{i}=\varepsilon_{\sigma^{-1}(i)},\quad\sigma\in
S_{m},\quad\varepsilon=(\varepsilon_{i})_{1\leq i\leq m}\in\{\pm\}^{m}.
\]
Let $\epsilon$ be the unique isomorphism of groups $\{\pm\}\rightarrow
\{0,1\}$. For $g\in G$, write%
\[
g\circ\varphi_{i}=\iota^{\epsilon(\varepsilon_{i})}\varphi_{\sigma^{-1}(i)}.
\]
Then $\rho_{\Phi}$ is the homomorphism $g\mapsto((\varepsilon_{i})_{i}%
,\sigma)$.

There is a natural action of $\{\pm\}^{m}\rtimes S_{m}$ on $\mathbb{R}{}^{m}$:%
\[
((\varepsilon_{i})_{1\leq i\leq m},\sigma)(x_{i})_{1\leq i\leq m}%
=(\varepsilon_{i}x_{\sigma^{-1}(i)})_{1\leq i\leq m}\text{.}%
\]
By composition, we get a linear representation of $G$ on $\mathbb{R}{}^{m}$,
also denoted $\rho_{\Phi}$. This acts transitively transitively on $a(2^{m})$,
and $\rho_{\Phi}(\iota)$ acts as $-1$.

The next statement is \cite{hazama2003}, 6.2, but the proof there is
incomplete.\footnote{It applies only to the simple subvariety of the abelian
variety of CM-type $(K,\Phi)$. Also the description (ibid. p632) of the
divisor classes is incorrect, and so the statement in Theorem 7.14 that
$A_{A(2^{n})}(G;K)$ has exotic Hodge classes when $n\geq3$ should be treated
with caution.}

\begin{proposition}
\label{t30d}Every simple CM abelian variety split by $K$ is isogenous to a
subvariety of $A(G,K,\rho_{\Phi})$.
\end{proposition}

\begin{proof}
Because of Proposition \ref{t30b}, it suffices to show that $A(G,K,\rho_{\Phi
})$ is isogenous to the abelian variety $A(S,S^{+})$ of (\ref{t30a}). For a
CM-type $\Phi^{\prime}$ on $K$, let $\varepsilon_{i}(\Phi^{\prime})$ equal $+$
or $-$ according as $\varphi_{i}\in\Phi^{\prime}$ or not. Then $\Phi^{\prime
}\mapsto(\varepsilon_{i}(\Phi^{\prime}))_{1\leq i\leq m}\colon S\rightarrow
\{\pm\}^{m}\simeq R(2^{m})$ is a bijection, sending $S^{+}$ onto $R(2^{m}%
)^{+}$. This map is $G$-equivariant, and so $A(S,S^{+})$ is isogenous to
$A(R(2^{m}),R(2^{m})^{+})\overset{\text{{\tiny def}}}{=}A(G,K,\rho_{\Phi}).$
\end{proof}

The following is the main theorem of \cite{hazama2002, hazama2003}.

\begin{theorem}
\label{t30e}In order to prove the Hodge conjecture for CM abelian varieties
over $\mathbb{C}{}$, it suffices to prove it in codimension $2.$
\end{theorem}

\begin{proof}
If the Hodge conjecture holds in codimension $2$, then Theorem \ref{t30c}
shows that it holds for the varieties $A(G,K,\rho_{\Phi})^{n}$, but
Proposition \ref{t30d} shows that every CM abelian $A$ is isogenous to a
subvariety of $A(G,K,\rho_{\Phi})^{n}$ for some $K$ and $n$. It is easy to see
that if the Hodge conjecture holds for an abelian variety $A$, then it holds
for any abelian subvariety (because it is an isogeny factor).
\end{proof}

\subsection{The Tate conjecture}

Theorems \ref{t30e} and \ref{t44} show that, in order to prove the Tate
conjecture for abelian varieties over $\mathbb{F}{}$, it suffices to prove the
Hodge conjecture in codimension $2$. The following is a more natural statement.

\begin{theorem}
\label{t29}In order to prove the Tate conjecture for abelian varieties over
$\mathbb{F}{}$, it suffices to prove it in codimension $2$.
\end{theorem}

More precisely, we shall show that if $T^{2}(A,\ell)$ holds for all abelian
varieties $A$ over $\mathbb{F}{}$ and some $\ell$, then $T^{r}(A,\ell)$ and
$E^{r}(A,\ell)$ hold for all abelian varieties $A$ over $\mathbb{F}{}$ and all
$r$ and $\ell$. We fix a $p$-adic prime $w$ of $\mathbb{Q}{}^{\mathrm{al}}$,
and use the same notations as in \S 7.

\begin{lemma}
\label{t29a}Let $A$ be an abelian variety over $\mathbb{Q}{}^{\mathrm{al}}$.
If $A$ is split by $K$, then so also is $A_{0}$. Conversely, every abelian
variety over $\mathbb{F}{}$ split by $K$ is isogenous to an abelian variety
$A_{0}$ with $A$ split by $K$.
\end{lemma}

\begin{proof}
Let $E$ be an \'{e}tale subalgebra of $\End^{0}(A)$ such that $[E\colon
\mathbb{Q}{}]=2\dim A$. Then $E$ is a maximal \'{e}tale subalgebra of
$\End^{0}(A_{0})$. Because $E$ is split by $K$, so also is $\End^{0}(A_{0})$.
The converse follows from \cite{tate1968}, Lemme 3.
\end{proof}

\begin{proof}
[of Theorem \ref{t29}]Let $A=A(S,S^{+})$ be the abelian variety in
(\ref{t30b}), which (see \S 7) we may regard as an abelian variety over
$\mathbb{Q}^{\mathrm{al}}$. Then $A_{0}$ is an abelian variety over
$\mathbb{F}{}$ split by $K$, and every simple abelian variety over
$\mathbb{F}{}$ split by $K$ is isogenous to an abelian subvariety of $A_{0}$
(by \ref{t30b}, \ref{t29a}). The inclusion $\End^{0}(A)\hookrightarrow
\End^{0}(A_{0})$ realizes $C(A_{0})$ as a $\mathbb{Q}{}$-subalgebra of $C(A)$,
and hence defines an inclusion $L(A_{0})\rightarrow L(A)$ of Lefschetz groups
(see \S 2). Consider the diagram%
\[
\xymatrix{
\MT(A)\ar@{^{(}->}[r]&L(A)\\
P(A_{0})\ar@{-->}[u]\ar@{^{(}->}[r]&L(A_{0})\ar@{^{(}->}[u]
}
\]
in which $MT(A)$ is the Mumford-Tate group of $A$ and $P(A_{0})$ is the
smallest algebraic subgroup of $L(A_{0})$ containing a Frobenius endomorphism
of $A_{0}$. Almost by definition, $\MT(A)$ is the largest algebraic subgroup
of $L(A)$ fixing the Hodge classes in $H_{B}^{2\ast}(A_{\mathbb{C}{}}%
^{n})(\ast)$ for all $n$, and so, for any prime $\ell$, $\MT(A)_{\mathbb{Q}%
{}_{\ell}}$ is the largest algebraic subgroup of $L(A)_{\mathbb{Q}{}_{\ell}}$
fixing the Hodge classes in $H_{\ell}^{2\ast}(A^{n})(\ast)$ for all $n$. On
the other hand, the classes in $H_{\ell}^{2\ast}(A_{0}^{n})(\ast)$ fixed by
$P(A_{0})_{\mathbb{Q}{}_{\ell}}$ are exactly the Tate classes. The
specialization isomorphism $H_{\ell}^{2\ast}(A)(\ast)\rightarrow H_{\ell
}^{2\ast}(A_{0})(\ast)$ is equivariant for the homomorphism $L(A_{0}%
)\rightarrow L(A)$. As Hodge classes map to Tate classes in $H_{\ell}^{2\ast
}(A)(\ast)$ (see \S 6), they map to Tate classes in $H_{\ell}^{2\ast}%
(A_{0})(\ast)$, and so they are fixed by $P(A_{0})_{\mathbb{Q}{}_{\ell}}$.
This shows that $P\left(  A_{0}\right)  _{\mathbb{Q}{}_{\ell}}\subset
\MT(A)_{\mathbb{Q}{}_{\ell}}$ (inside $L(A)_{\mathbb{Q}{}_{\ell}}$), and so
$P(A_{0})\subset\MT(A)$ (inside $L(A)$). The following is the main technical
result of \cite{milne1999lc} (Theorem 6.1): \bquote Assume that $K$ contains a
quadratic imaginary field. Then the algebraic subgroup $\MT(A)$ and $L(A_{0})$
of $L(A)$ intersect in $P(A_{0})$.\equote

We now enlarge $K$ so that it contains a quadratic imaginary field. Corollary
\ref{t19} allows us to choose $\ell$ so that $E(A_{0},\ell)$ holds. Let $G$ be
the algebraic subgroup of $L(A_{0})_{\mathbb{Q}{}_{\ell}}$ fixing the
algebraic classes in $H_{\ell}^{2\ast}(A_{0}^{n})(\ast)$ for all $n$. If the
Tate conjecture holds in codimension $2$, then $G$ fixes the Tate classes in
$H_{\ell}^{4}(A_{0}^{n})(2)$ (all $n$); therefore (by Theorem \ref{t30c}), it
fixes the Hodge classes in $H_{\ell}^{2\ast}(A^{n})(\ast)$ (all $n$), and so
$G\subset\MT(A)_{\mathbb{Q}{}_{\ell}}\cap L(A_{0})_{\mathbb{Q}{}_{\ell}%
}=P(A_{0})_{\mathbb{Q}{}_{\ell}}$; therefore, $G$ fixes all Tate classes in
$H_{\ell}^{2\ast}(A_{0}^{n})(\ast)$ (all $n$), which shows that the space of
Tate classes in $H_{\ell}^{2\ast}(A_{0}^{n})(\ast)$ (all $n$) is spanned by
the algebraic classes (apply Theorem \ref{t15}). Therefore, the equivalent
statements in Theorem \ref{t2} hold for $A_{0}^{n}$ for all $r$ and $n$. It
follows that the same is true of every abelian subvariety of some power
$A_{0}$, which includes all abelian varieties over $\mathbb{F}{}$ split by $K$
(apply Lemma \ref{t29a}). Since every abelian variety over $\mathbb{F}{}$ is
split by some CM-field, this completes the proof.
\end{proof}

Optimists will now try to prove the Hodge conjecture in codimension $2$ for CM
abelian varieties, or, what may (or may not) be easier, the Tate conjecture in
codimension $2$ for abelian varieties over $\mathbb{F}{}$. Pessimists will try
to prove the opposite. Others may prefer to look at the questions in \S 10.

\section{The Hodge standard conjecture}

Let $k$ be an algebraically closed field, and let $H_{W}$ be a Weil cohomology
theory on the varieties over $k$. For a variety $X$ over $k$, let
$\mathcal{A}{}_{W}^{r}(X)$ be the $\mathbb{Q}{}$-subspace of $H_{W}%
^{2r}(X)(r)$ spanned by the classes of algebraic cycles. Let $\xi\in H_{W}%
^{2}(X)(1)$ be the class of a hyperplane section of $X$, and let $L\colon
H_{W}^{i}(X)\rightarrow H_{W}^{i+2}(X)(1)$ be the map $\cup\xi$. The
\emph{primitive part }$\mathcal{A}{}_{W}^{r}(X)_{\text{prim}}$ of
$\mathcal{A}{}_{W}^{r}(X)$ is defined to be%
\[
\mathcal{A}{}_{W}^{r}(X)_{\text{prim}}=\{z\in\mathcal{A}{}_{W}^{r}(X)\mid
L^{\dim(X)-2r+1}z=0\}\text{.}%
\]

\begin{conjecture}
[Hodge standard]Let $d=\dim X$. For $2r\leq d$, the symmetric bilinear form%
\[
(x,y)\mapsto(-1)^{r}x\cdot y\cdot\xi^{d-2r}\colon\mathcal{A}{}_{W}%
^{r}(X)_{\text{prim}}\times\mathcal{A}{}_{W}^{r}(X)_{\text{prim}}%
\rightarrow\mathcal{A}_{W}^{d}(X)\simeq\mathbb{Q}{}{}%
\]
is positive definite (\cite{grothendieck1969s}, \textit{Hdg(}$X$\textit{)}).
\end{conjecture}

The next theorem is an abstract version of the main theorem of
\cite{milne2002p}.

\begin{theorem}
\label{t31}The Hodge standard conjecture holds for every good theory of
rational Tate classes.
\end{theorem}

In more detail, let $\mathcal{(\mathcal{R}{}}^{\ast}(X))_{X\in\mathcal{S}{}}$
be a good theory of rational Tate classes. For $X\in\mathcal{S}{}$, the
cohomology class $\xi$ of a hyperplane section of $X$ lies in $\mathcal{R}%
{}^{1}(X)$, and we can define $\mathcal{R}{}^{r}(X)_{\text{prim}}$ and the
pairing on $\mathcal{R}^{r}(X)_{\text{prim}}$ by the above formulas. The
theorem states that this pairing%
\[
\mathcal{R}^{r}(X)_{\text{prim}}\times\mathcal{R}^{r}(X)_{\text{prim}%
}\rightarrow\mathbb{Q}{}%
\]
is positive definite.

For the proof, see \cite{milne2007rtc}. I list one application of this theorem.

\begin{theorem}
\label{t32}If there exists a good theory of rational Tate classes for which
all algebraic classes are rational Tate classes, then the Hodge standard
conjecture holds.
\end{theorem}

\begin{proof}
The bilinear form on $\mathcal{R}{}^{r}(X)_{\text{prim}}$ restricts to the
correct bilinear form on $\mathcal{A}{}^{r}(X)_{\text{prim}}$. If the first is
positive definite, then so is the second, which implies that the form on
$\mathcal{A}{}_{W}^{r}(X)_{\text{prim}}$ is positive definite for any Weil
cohomology theory $H_{W}$.
\end{proof}

\begin{aside}
\label{t33}Let $\mathcal{S}{}_{0}$ be the smallest class satisfying (*) and
let $\mathcal{S}{}$ be a second (possibly larger) class. If the Hodge
conjecture holds for CM abelian varieties, then the family $(\mathcal{A}{}%
^{r}(X))_{X\in\mathcal{S}{}_{0}}$ is a good theory of rational Tate classes
for $\mathcal{S}{}_{0}$; if moreover, the Tate conjecture holds for all
varieties in $\mathcal{S}{}$, then $(\mathcal{A}{}^{r}(X))_{X\in\mathcal{S}{}%
}$ is a good theory of rational Tate classes for $\mathcal{S}{}$. However, the
Tate conjecture alone does not imply that $(\mathcal{A}{}^{r}(X))_{X\in
\mathcal{S}{}}$ is a good theory of rational Tate classes on $\mathcal{S}{}$;
in particular, we don't know that the Tate conjecture implies the Hodge
standard conjecture. Thus, in some respects, the existence of a good theory of
rational Tate classes is a stronger statement than the Tate conjecture for
varieties over $\mathbb{F}{}$.
\end{aside}

\section{On the existence of a good theory of rational Tate classes}

I consider this only for the smallest class $\mathcal{S}{}_{0}$ satisfying
(*), which, I recall, contains the abelian varieties.

\begin{conjecture}
[Rationality Conjecture]\label{t35}Let $A$ be a CM abelian variety over
$\mathbb{Q}{}^{\mathrm{al}}$. The product of the specialization to $A_{0}$ of
any Hodge class on $A$ with any Lefschetz class on $A_{0}$ of complementary
dimension lies in $\mathbb{Q}{}$.
\end{conjecture}

In more detail, a Hodge class on $A$ is an element of $\gamma$ of
$H_{\mathbb{A}{}}^{2\ast}(A)(\ast)$ and its specialization $\gamma_{0}$ is an
element of $H_{\mathbb{A}{}}^{2\ast}(A_{0})(\ast)$. Thus the product
$\gamma_{0}\cdot\delta$ of $\gamma_{0}$ with a Lefschetz class of
complementary dimension $\delta$ lies in
\[
H_{\mathbb{A}{}}^{2d}(A_{0})(d)\simeq\mathbb{A}{}_{f}^{p}\times\mathbb{Q}%
{}_{w}^{\mathrm{al}},\quad d=\dim(A).
\]
The conjecture says that it lies in $\mathbb{Q}{}\subset\mathbb{A}{}_{f}%
^{p}\times\mathbb{Q}{}_{w}^{\mathrm{al}}$. Equivalently, it says that the
$l$-component of $\gamma_{0}\cdot\delta$ is a rational number independent of
$l$.

\begin{remark}
\label{t36}(a) The conjecture is true for a particular $\gamma$ if $\gamma
_{0}$ is algebraic. Therefore, the conjecture is implied by the Hodge
conjecture for CM abelian varieties (or even by the weaker statement that the
Hodge classes specialize to algebraic classes).

(b) The conjecture is true for a particular $\delta$ if it lifts to a rational
cohomology class on $A$. In particular, the conjecture is true if $A_{0}$ is
ordinary and $A$ is its canonical lift (because then all Lefschetz classes on
$A_{0}$ lift to Lefschetz classes on $A$).
\end{remark}

For an abelian variety $A$ over $\mathbb{F}{}$, let $\mathcal{L}{}^{\ast}(A)$
be the $\mathbb{Q}{}$-subalgebra of $H_{\mathbb{A}{}}^{2\ast}(A)(\ast)$
generated by the divisor classes, and call its elements the \emph{Lefschetz
classes} on $A$.

\begin{definition}
\label{t36a}Let $A$ be an abelian variety over $\mathbb{Q}{}^{\mathrm{al}}$
with good reduction to an abelian variety $A_{0}$ over $\mathbb{F}{}$. A Hodge
class $\gamma$ on $A$ is \emph{locally }$w$\emph{-Lefschetz} if its image
$\gamma_{0}$ in $H_{\mathbb{A}{}}^{2\ast}(A_{0})(\ast)$ is in the
$\mathbb{A}{}$-span of the Lefschetz classes, and it is $w$\emph{-Lefschetz
}if $\gamma_{0}$ is itself Lefschetz.
\end{definition}

\begin{conjecture}
[Weak Rationality Conjecture]\label{t36b}Let $A$ be an abelian variety over
$\mathbb{Q}{}^{\mathrm{al}}$ with good reduction to an abelian variety $A_{0}$
over $\mathbb{F}{}$. Every locally $w$-Lefschetz Hodge class is itself $w$-Lefschetz.
\end{conjecture}

\begin{theorem}
\label{t36c}The following statements are equivalent:

\begin{enumerate}
\item The rationality conjecture holds for all CM abelian varieties over
$\mathbb{\mathbb{Q}{}}^{\mathrm{al}}{}$.

\item The weak rationality conjecture holds for all CM abelian varieties over
$\mathbb{Q}{}^{\mathrm{al}}$.

\item There exists a good theory of rational Tate classes on abelian varieties
over $\mathbb{F}{}$.
\end{enumerate}
\end{theorem}

\begin{proof}
(a)$\implies$(b): Choose a $\mathbb{Q}{}$-basis $e_{1},\ldots,e_{t}$ for the
space of Lefschetz classes of codimension $r$ on $A_{0}$, and let
$f_{1},\ldots,f_{t}$ be the dual basis for the space of Lefschetz classes of
complementary dimension (here we use \cite{milne1999lc}, 5.2, 5.3). If
$\gamma$ is a locally $w$-Lefschetz class of codimension $r$, then $\gamma
_{0}=\sum c_{i}e_{i}$ for some $c_{i}\in\mathbb{A}{}$. Now%
\[
\langle\gamma_{0}\cdot f_{j}\rangle=c_{j},
\]
which (a) implies lies in $\mathbb{Q}{}$.

(b)$\implies$(c): See \cite{milne2007rtc}.

(c)$\implies$(a): If there exists a good theory $\mathcal{R}{}$ of rational
Tate classes, then certainly the rationality conjecture is true, because then
$\gamma_{0}\cdot\delta\in\mathcal{R}{}^{2d}\simeq\mathbb{Q}{}$.
\end{proof}

\subsection{Two questions}

\begin{question}
\label{t38}Let $A$ be a CM abelian variety over $\mathbb{Q}{}^{\mathrm{al}}$,
let $\gamma$ be a Hodge class on $A$, and let $\delta$ be a divisor class on
$A_{0}$. Does $(A_{0},\gamma_{0},\delta)$ always lift to characteristic zero?
That is, does there always exist a CM abelian variety $A^{\prime}$ over
$\mathbb{Q}{}^{\mathrm{al}}$, a Hodge class $\gamma^{\prime}$ on $A^{\prime}$,
a divisor class $\delta^{\prime}$ on $A^{\prime}$ and an isogeny
$A_{0}^{\prime}\rightarrow A_{0}$ sending $\gamma_{0}^{\prime}$ to $\gamma
_{0}$ and $\delta_{0}^{\prime}$ to $\delta$?
\end{question}

\begin{proposition}
\label{t38p}If Question \ref{t38} has a positive answer, then the rationality
conjecture holds for all CM abelian varieties.
\end{proposition}

\begin{proof}
Let $\gamma$ be a Hodge class on a CM abelian variety $A$ of dimension $d$
over $\mathbb{Q}{}^{\mathrm{al}}$. If $\gamma$ has dimension $\leq1$, then it
is algebraic and so satisfies the rationality conjecture. We shall proceed by
induction on the codimension of $\gamma$. Assume $\gamma$ has dimension
$r\geq2$, and let $\delta$ be a Lefschetz class of dimension $d-r$. We may
suppose that $\delta=\delta_{1}\cdot\delta_{2}\cdots$ where $\delta_{1}%
,\delta_{2},\ldots$ are divisor classes. Apply (\ref{t38}) to $(A,\gamma
,\delta)$. Then $\gamma^{\prime}\cdot\delta_{1}^{\prime}$ is a Hodge class on
$A^{\prime}$ of codimension $r-1$, and%
\[
\gamma_{0}\cdot\delta\in(\gamma^{\prime}\cdot\delta_{1}^{\prime})_{0}%
\cdot\delta_{2}\cdot\cdots\cdot\delta_{d-r}\mathbb{Q}{}\subset\mathbb{Q}%
{}\text{.}%
\]

\end{proof}

A pair $(A,\nu)$ consisting of an abelian variety $A$ over $\mathbb{C}{}$ and
a homomorphism $\nu$ from a CM field $E$ to $\End^{0}(A)$ is said to be of
\emph{Weil type} if the tangent space to $A$ at $0$ is a free $E\otimes
_{\mathbb{Q}{}}k{}$-module. For such a pair $(A,\nu)$, the space
\[
W^{d}(A,\nu)\overset{\text{{\tiny def}}}{=}\bigwedge\nolimits_{E}^{d}%
H^{1}(A,\mathbb{Q}{})\subset H^{d}(A,\mathbb{Q}{}),\text{ where }d=\dim
_{E}H^{1}(A,\mathbb{Q}{}),
\]
consists of Hodge classes (\cite{deligne1982}, 4.4). When $E$ is quadratic
over $\mathbb{Q}{}$, the spaces $W^{d}$ were studied by Weil (1977), and for
this reason its elements are called \emph{Weil classes}. A \emph{polarization}
of an abelian variety $(A,\nu)$ of Weil type is a polarization of $A$ whose
Rosati involution stabilizes $E$ and induces complex conjugation on it. There
then exists an $E$-hermitian form $\phi$ on $H_{1}(A,\mathbb{Q}{})$ and an
$f\in E^{\times}$ with $\bar{f}=-f$ such that $\psi(x,y)\overset
{\text{{\tiny def}}}{=}\Tr_{E/\mathbb{Q}{}}(f\phi(x,y))$ is a Riemann form for
$\lambda$ (ibid. 4.6). We say that the Weil classes on $(A,\nu)$ are
\emph{split} if there exists a polarization of $(A,\nu)$ for which the
$E$-hermitian form $\phi$ is split (i.e., admits a totally isotropic subspace
of dimension $\dim_{E}H_{1}(A,\mathbb{Q}{})/2$).

\begin{question}
\label{t39}Is it possible to prove the weak rationality conjecture for split
Weil classes on CM abelian variety by considering the families of abelian
varieties considered in \cite{deligne1982}, proof of 4.8, and
\cite{andre2006b}, \S 3?
\end{question}

A positive answer to this question implies the weak rationality conjecture
because of the following result of Andre (1992)\nocite{andre1992} (or the
results of Deligne 1982, \S 5).

\begin{quote}
\label{t45}Let $A$ be a CM abelian variety over $\mathbb{C}{}$. Then there
exist CM abelian varieties $B_{i}$ and homomorphisms $A\rightarrow B_{i}$ such
that every Hodge class on $A$ is a linear combination of the inverse images of
split Weil classes on the $B_{i}$.
\end{quote}

\noindent In the spirit of \cite{weil1967}, I leave the questions as exercises
for the interested reader.

\begin{aside}
\label{t40}In the paper in which they state their conjecture concerning the
structure of the points on a Shimura variety over a finite field, Langlands
and Rapoport prove the conjecture for some simple Shimura varieties of
PEL-type under the assumption of the Hodge conjecture for CM-varieties, the
Tate conjecture for abelian varieties over finite fields, and the Hodge
standard conjecture for abelian varieties over finite fields. I've proved that
the first of these conjectures implies the other two (see \ref{t44} and
\ref{t32}), and so we have gone from needing three conjectures to needing only
one. A proof of the rationality conjecture would eliminate the need for the
remaining conjecture. Probably we can get by with much less, but having come
so far I would like to finish it off with no fudges.
\end{aside}

\begin{aside}
Readers of the Wall Street Journal on August 1, 2007, were excited to find a
headline on the front page of Section B directing them to a column on
\textquotedblleft The Secret Life of Mathematicians\textquotedblright. The
column was about the workshop, and included the following paragraph:\bquote
Progress, though, was made. V. Kumar Murty, of the University of Toronto, said
that as a result of the sessions, he'd be pursuing a new line of attack on
Tate. It makes use of ideas of the J.S. Milne of Michigan, who was also in
attendance, and involves Abelian varieties over finite fields, in case you
want to get started yourself.\equote This becomes more-or-less correct when
you replace \textquotedblleft Tate\textquotedblright\ with the
\textquotedblleft weak rationality conjecture\textquotedblright.\footnote{The
column on the WSJ website is available only to subscribers, but there is a
summary of it on the MAA site at
\url{http://mathgateway.maa.org/do/ViewMathNews?id=143}..}
\end{aside}

\bsmall



\bibliographystyle{cbe}
\bibliography{D:/MData/Bib/refs}

\bigskip
\noindent Mathematics Department, University of Michigan, Ann Arbor, MI 48104, USA

\noindent\url{www.jmilne.org/math/}

\esmall
\end{document}